\documentclass[11pt]{article}
\usepackage{amsfonts}
\usepackage{color}
\usepackage{amsmath,amssymb,comment}
\newtheorem{theo}{Theorem}[section]
\newtheorem{lem}[theo]{Lemma}

\newtheorem{prop}[theo]{Proposition}

\newtheorem{defi}[theo]{Definition}

\newcommand{\mysection}[1]{\section{#1} \setcounter{equation}{0}}
\newcommand{\proof}{{\sc Proof.} \quad}
\newcommand{\proofc}{{\sc Proof} \ }
\newcommand{\be}{\begin{equation} \label}
\newcommand{\ee}{\end{equation}}
\newcommand{\bea}{\begin{eqnarray}\label}
\newcommand{\eea}{\end{eqnarray}}
\newcommand{\bas}{\begin{eqnarray*}}
\newcommand{\eas}{\end{eqnarray*}}
\newcommand{\bit}{\begin{itemize}}
\newcommand{\eit}{\end{itemize}}
\newcommand{\qed}{\hfill$\Box$ \vskip.2cm}
\newcommand{\nn}{\nonumber}
\newcommand{\R}{\mathbb{R}}
\newcommand{\N}{\mathbb{N}}
\newcommand{\pO}{\partial\Omega}

\newcommand{\eps}{\varepsilon}

\newcommand{\supp}{{\rm supp} \, }

\newcommand{\wto}{\rightharpoonup}
\newcommand{\wsto}{\stackrel{\star}{\rightharpoonup}}
\newcommand{\hra}{\hookrightarrow}
\newcommand{\io}{\int_\Omega}

\newcommand{\del}{\delta}

\newcommand{\pa}{\partial}
\newcommand{\bom}{\overline{\Omega}}
\newcommand{\Om}{\Omega}

\newcommand{\wh}{\widehat}

\newcommand{\hs}{\hspace*}
\newcommand{\sm}{\setminus}

\newcommand{\vp}{\varphi}
\newcommand{\lbal}{\left\{ \begin{array}{l}}
\newcommand{\lball}{\left\{ \begin{array}{ll}}
\newcommand{\ear}{\end{array} \right.}

\newcommand{\abs}{\\[5pt]}

\newcommand{\ueps}{u_\eps}
\newcommand{\veps}{v_\eps}

\newcommand{\yeps}{y_\eps}

\newcommand{\Teps}{\Theta_\eps}

\newcommand{\vepsx}{v_{\eps x}}
\newcommand{\vepsxx}{v_{\eps xx}}
\newcommand{\vepsxxx}{v_{\eps xxx}}
\newcommand{\vepsxxxx}{v_{\eps xxxx}}
\newcommand{\vepst}{v_{\eps t}}
\newcommand{\uepsx}{u_{\eps x}}
\newcommand{\uepsxx}{u_{\eps xx}}
\newcommand{\uepsxxx}{u_{\eps xxx}}
\newcommand{\uepst}{u_{\eps t}}
\newcommand{\Tepsx}{\Theta_{\eps x}}
\newcommand{\Tepsxx}{\Theta_{\eps xx}}
\newcommand{\Tepst}{\Theta_{\eps t}}

\newcommand{\zd}{\zeta_\delta}
\newcommand{\whve}{\wh{v}_\eps}

\newcommand{\whvex}{\wh{v}_{\eps x}}

\newcommand{\whue}{\wh{u}_\eps}

\newcommand{\whuex}{\wh{u}_{\eps x}}

\newcommand{\whTe}{\wh{\Theta}_\eps}

\newcommand{\whTex}{\wh{\Theta}_{\eps x}}

\newcommand{\whv}{\wh{v}}
\newcommand{\whu}{\wh{u}}
\newcommand{\whux}{\wh{u}_x}
\newcommand{\whT}{\wh{\Theta}}
\newcommand{\whTx}{\wh{\Theta}_x}
\begin{document}
\allowdisplaybreaks
\title{
A simple model for one-dimensional nonlinear thermoelasticity:\\
Well-posedness in rough-data frameworks}
\author{
Michael Winkler\footnote{michael.winkler@math.uni-paderborn.de}\\
{\small Universit\"at Paderborn, Institut f\"ur Mathematik}\\
{\small 33098 Paderborn, Germany} }
\date{}
\maketitle
\begin{abstract}
\noindent 
In an open bounded interval $\Om$, the problem
\bas
	\lball
	u_{tt} = u_{xx} - \big(f(\Theta)\big)_x,
	\qquad & x\in\Om, \ t>0, \\[1mm]
	\Theta_t = \Theta_{xx} - f(\Theta) u_{xt},
	\qquad & x\in\Om, \ t>0, 
	\ear
\eas
is considered under the boundary conditions $u|_{\pO}=\Theta_x|_{\pO}=0$, 
and for $f\in C^2([0,\infty))$ satisfying $f(0)=0$,
$f'>0$ on $[0,\infty)$ and $f'\in W^{1,\infty}((0,\infty))$.\abs
In the sense of unconditional global existence, uniqueness and continuous dependence, 
this problem is shown to be well-posed within ranges of initial data merely satisfying
\bas
	u_0\in W_0^{1,2}(\Om),
	\quad
	u_{0t} \in L^2(\Om)
	\quad \mbox{and} \quad
	\Theta_0 \in L^2(\Om)
	\mbox{ with $\Theta\ge 0$ a.e.~in $\Om$,}
\eas
and in classes of solutions fulfilling
\bas
	\lbal
	u\in C^0([0,\infty);W_0^{1,2}(\Om)), \\[1mm]
	u_t \in C^0([0,\infty);L^2(\Om))
	\qquad \mbox{and} \\[1mm]
	\Theta\in C^0([0,\infty);L^2(\Om)) \cap L^2_{loc}([0,\infty);W^{1,2}(\Om)).
	\ear
\eas
\noindent {\bf Key words:} nonlinear acoustics; thermoelasticity; uniqueness; continuous dependence\\
{\bf MSC 2020:} 74H20 (primary); 74F05, 35L05, 35D30 (secondary)
\end{abstract}
%
%
%
%
%
%
%
\newpage
\section{Introduction}\label{intro}
In mathematical descriptions of nonlinear thermoelastic interaction in one-dimensional materials, the hyperbolic-parabolic system
\be{00}
	\lball
	u_{tt} = u_{xx} - \Theta_x, \\[1mm]
	\Theta_t = \Theta_{xx} - \Theta u_{xt},
	\ear
\ee
arises as a minimal representative within classes of more general models (\cite{slemrod}, \cite{hrusa_tarabek}).
In line with pioneering discoveries on the occurrence of finite-time $C^2$-blow-up phenomena in related systems
(\cite{dafermos_hsiao_BU}, \cite{racke_bu}, \cite{hrusa_messaoudi}), early literature concerned with existence
theories for (\ref{00}) and more complex variants concentrated on the construction either of local-in-time classical
solutions (\cite{slemrod}, \cite{racke88}), or on global smooth solutions evolving from suitably small perturbations 
of homogeneous states (\cite{slemrod}, \cite{hrusa_tarabek}, \cite{racke88}, \cite{racke_shibata}, \cite{racke_shibata_zheng}, 
\cite{kim}, \cite{jiang1990}); results on related linear problems can be found in \cite{dafermos1968}, \cite{koch}, 
\cite{jiang_racke_book} and \cite{yaguang_wang}, for instance.
Global solutions far from equilibrium seem to have been considered mainly for relatives of (\ref{00}) which 
additionally account for viscosity effects and hence for a supplementary dissipative mechanism
(\cite{racke_zheng}, \cite{hsiao_luo}).\abs
In contrast to this, global large-data solutions to the purely thermoelastic nonlinear system (\ref{00}) 
seem to have been obtained only recently:
In \cite{cieslak_SIMA}, (\ref{00}) has been considered in a bounded real interval $\Om$ 
and under homogeneous boundary conditions of Dirichlet type for
the displacement field $u$, and of Neumann type for the temperature variable $\Theta$, and under the assumption that
the initial data be regular enough fulfilling
\be{ireg}
	u_0\in W^{2,2}(\Om) \cap W_0^{1,2}(\Om),
	\qquad 
	u_{0t} \in W_0^{1,2}(\Om)
	\qquad \mbox{and} \qquad
	\Theta_0\in W^{1,2}(\Om)
	\mbox{ with $\Theta_0>0$ in } \bom,
\ee
a corresponding initial-boundary value problem with initial condition $(u,u_t,\Theta)|_{t=0}=(u_0,u_{0t},\Theta_0)$
has been found to admit a unique global solution $(u,\Theta)$ with, inter alia,
$u\in L^\infty((0,\infty);W^{2,2}(\Om))$, $u_t\in L^\infty((0,\infty);W^{1,2}(\Om))$ 
and $\Theta\in L^\infty((0,\infty);W^{1,2}(\Om))$ (cf.~\cite[Definition 1.2 and Theorem 1.3]{cieslak_SIMA} for a more
precise statement); beyond this, in \cite{cieslak_MAAN} the large time behavior of all these solutions has been found to 
be determined by stabilization toward constant equilibria.
We remark that for multi-dimensional versions and relatives of (\ref{00}), large-data solutions have been constructed 
in various particular modeling contexts of partially considerable complexity, 
but predominantly within suitably generalized concepts of solvability
(cf.~\cite{roubicek}, \cite{mielke_roubicek}, \cite{owczarek_wielgos}, \cite{cieslak_muha_trifunovic},
\cite{roubicek_SIMA10}, \cite{blanchard_guibe} and \cite{pawlow_zajaczkowski_SIMA13} for a small selection).\abs
While the set of initial data satisfying (\ref{ireg}) does contain functions of arbitrary size in the respective spaces,
this class is significantly narrower than that of all data which are compatible with the natural energy identity
\be{energy}
	\frac{d}{dt} \bigg\{ \frac{1}{2} \io u_t^2 + \frac{1}{2} \io u_x^2 + \io \Theta \bigg\}
	= 0
\ee
formally associated with (\ref{00}).
After all, under the requirement that merely
\be{01}
	u_0\in W_0^{1,2}(\Om),
	\qquad 
	u_{0t} \in L^2(\Om)
	\qquad \mbox{and} \qquad
	\Theta_0\in L^1(\Om)
	\mbox{ is such that $\Theta_0\ge 0$ a.e.~in } \Om,
\ee
for a problem slightly more general than (\ref{00}) (see (\ref{0}) below) a statement on global existence of at least one
solution could recently be derived in a suitable framework of weak solvability (see Definition \ref{dw} and \cite{win_existence}),
and a statement on large time convergence to homogeneous states can be found in \cite{win_conv}.
It seems unknown, however, whether in these energy-minimal settings such weak solutions must be unique.\abs
{\bf Main results.} \quad
The intention of the present manuscript is to address this latter issue, and to identify a scenario capable of enforcing 
genuine well-posedness, with a particular focus on the question how far (\ref{00}) is well-behaved in this regard also
in the presence of large and possibly even discontinuous initial strain distributions $u_{0x}$ admissible, e.g., in (\ref{01}).
To investigate this in the context of the problem class addressed in \cite{win_existence}, 
in an open bounded real interval $\Om$ we will subsequently consider
\be{0}
	\lball
	u_{tt} = u_{xx} - \big(f(\Theta)\big)_x,
	\qquad & x\in\Om, \ t>0, \\[1mm]
	\Theta_t = \Theta_{xx} - f(\Theta) u_{xt},
	\qquad & x\in\Om, \ t>0, \\[1mm]
	u=0, \quad \Theta_x=0,
	\qquad & x\in\pO, \ t>0, \\[1mm]
	u(x,0)=u_0(x), \quad u_t(x,0)=u_{0t}(x), \quad \Theta(x,0)=\Theta_0(x),
	\qquad & x\in\Om.
	\ear
\ee
Here, choosing $f\equiv id$ leads to a reduction to the model (\ref{00})
which corresponds to one-dimensional bodies with free energy given by 
${\mathcal{E}}=\frac{1}{2} u_x^2 - \Theta u_x - \Theta \ln \Theta + \Theta$, 
while (\ref{0}) with more general $f$ can be regarded as a limit 
of a more complex and not necessarily hyperbolic-parabolic model obtained on choosing
${\mathcal{E}}=\frac{1}{2} u_x^2 - f(\Theta) u_x - \Theta \ln \Theta + \Theta$ within certain small-strain ranges
(\cite{hrusa_tarabek}, \cite{win_existence}).\abs
The conceptual basis for our considerations is formed by the following specification of solvability,
here yet quite moderate with regard to aspects of regularity.
\begin{defi}\label{dw}
  Let $\Om\subset\R$ be a bounded open interval, let 
  $f\in C^1([0,\infty))$,
  let
  $u_0\in L^1(\Om)$, $u_{0t}\in L^1(\Om)$ and $\Theta_0\in L^1(\Om)$, and let $T>0$.
  Then by a {\em weak solution} of (\ref{0}) in $\Om\times (0,T)$ we mean a pair $(u,\Theta)$ of functions
  \be{w1}
	\lbal
	u\in C^0([0,T];L^1(\Om)) \cap L^1((0,T);W_0^{1,1}(\Om))
	\qquad \mbox{and} \\[1mm]
	\Theta\in L^1((0,T);W^{1,1}(\Om))
	\ear
  \ee
  such that $\Theta\ge 0$ a.e.~in $\Om\times (0,T)$, that
  \be{w3}
	\big\{ u_t \, , \, f'(\Theta) \Theta_x \, , \, f'(\Theta) \Theta_x u_t \, , \, f(\Theta) u_t \big\}
	\subset L^1(\Om\times (0,T)),
  \ee
  and that
  \bea{wu}
	\int_0^T \io u \vp_{tt}
	- \io u_{0t} \vp(\cdot,0)
	+ \io u_0 \vp_t(\cdot,0)
	= -  \int_0^T \io u_x \vp_x
	- \int_0^T \io f'(\Theta) \Theta_x \vp
  \eea
  for all $\vp\in C_0^\infty(\Om\times [0,T))$, as well as
  \bea{wt}
	- \int_0^T \io \Theta \vp_t - \io \Theta_0 \vp(\cdot,0)
	= - \int_0^T \io \Theta_x \vp_x
	+ \int_0^T \io f'(\Theta) \Theta_x u_t \vp
	+ \int_0^T \io f(\Theta) u_t \vp_x
  \eea
  for each $\vp\in C_0^\infty(\bom\times [0,T))$.
  If $(u,\Theta):\Om\times (0,\infty) \to\R^2$ is a weak solution of (\ref{0}) in $\Om\times (0,T)$ for all $T>0$,
  then $(u,\Theta)$ will be called a {\em global weak solution} of (\ref{0}).
\end{defi}
Our main result now asserts that 
under assumptions on the initial data which are identical to the minimal energy-consistent
properties in (\ref{01}) with respect to the mechanical part, and which are only slightly stronger with regard to
the temperature distribution, even a uniquely determined global weak solution can always be found which actually enjoys
regularity features significantly stronger than those required in the above definition:
\begin{theo}\label{theo99}
  Let $\Om\subset\R$ be an open bounded interval, suppose that
  \be{f}
	f\in C^2([0,\infty)) 
	\mbox{ is such that $f(0)=0$ and $f'>0$ on $[0,\infty)$ as well as $f'\in W^{1,\infty}((0,\infty))$,}
  \ee
  and assume that
  \be{init}
	u_0\in W_0^{1,2}(\Om),
	\quad
	u_{0t} \in L^2(\Om)
	\quad \mbox{and} \quad
	\Theta_0 \in L^2(\Om)
	\mbox{ is such that $\Theta\ge 0$ a.e.~in $\Om$.}
  \ee
  Then there exists precisely one global weak solution $(u,\Theta)=:S(u_0,u_{0t},\Theta_0)$
  of (\ref{0}) which has the additional property that 
  \be{reg}
	\lbal
	u\in C^0([0,\infty);W_0^{1,2}(\Om)), \\[1mm]
	u_t \in C^0([0,\infty);L^2(\Om))
	\qquad \mbox{and} \\[1mm]
	\Theta\in C^0([0,\infty);L^2(\Om)) \cap L^2_{loc}([0,\infty);W^{1,2}(\Om)).
	\ear
  \ee
\end{theo}
In fact, these solutions furthermore depend continuously on the initial data in the following sense.
\begin{prop}\label{prop98}
  Let $\Om\subset\R$ be an open bounded interval, and assume (\ref{f}).
  Then the mapping $S$ defined in Theorem \ref{theo99} acts as a continuous operator 
  \bas
	& & \hs{-12mm}
	S: \ W_0^{1,2}(\Om)\times L^2(\Om) \times L^2(\Om) \nn\\
	&\to& C^0_{loc}([0,\infty);W_0^{1,2}(\Om)) \times C^0_{loc}([0,\infty);L^2(\Om))
  	\times \big( C^0_{loc}([0,\infty);L^2(\Om)) \cap L^2_{loc}([0,\infty);W^{1,2}(\Om))\big).
  \eas
  That is, if
  $(u_0^{(k)})_{k\in\N} \subset W_0^{1,2}(\Om)$, $(u_{0t}^{(k)})_{k\in\N} \subset L^2(\Om)$ and 
  $(\Theta_0^{(k)})_{k\in\N} \subset L^2(\Om)$
  are such that $\Theta_0^{(k)} \ge 0$ a.e.~in $\Om$ for all $k\in\N$, and that as $k\to\infty$ we have
  \be{98.1}
	\lbal
	u_0^{(k)} \to u_0 
	\qquad \mbox{in } W_0^{1,2}(\Om), \\[1mm]
	u_{0t}^{(k)} \to u_{0t}
	\qquad \mbox{in } L^2(\Om)
	\qquad \mbox{and} \\[1mm]
	\Theta_0^{(k)} \to \Theta_0
	\qquad \mbox{in } L^2(\Om)
	\ear
  \ee
  with some $u_0\in W_0^{1,2}(\Om), u_{0t} \in L^2(\Om)$ and $\Theta_0\in L^2(\Om)$,
  then the corresponding global weak solutions $(u^{(k)},\Theta^{(k)}):=S(u_0^{(k)},u_{0t}^{(k)},\Theta_0^{(k)})$, 
  $k\in\N$, and $(u,\Theta):=S(u_0,u_{0t},\Theta_0)$ of (\ref{0}) 
  according to Theorem \ref{theo99} satisfy 
  \be{98.2}
	\lbal
	u^{(k)} \to u
	\qquad \mbox{in } C^0_{loc}([0,\infty);W_0^{1,2}(\Om)), \\[1mm]
	u_t^{(k)} \to u_t
	\qquad \mbox{in } C^0_{loc}([0,\infty);L^2(\Om))
	\qquad \mbox{and} \\[1mm]
	\Theta^{(k)} \to \Theta
	\qquad \mbox{in } C^0_{loc}([0,\infty);L^2(\Om)) \cap L^2_{loc}([0,\infty);W^{1,2}(\Om))
	\ear
  \ee
  as $k\to\infty$.
\end{prop}
{\bf Main ideas.} \quad
  An essential part of our analysis will be related to the identification of appropriate
  frameworks in which inequalities in the style of
  \bea{02}
	& & \hs{-16mm}
	\sup_{t\in (0,T)} \bigg\{
	\io |u_t(\cdot,t)-\whu_t(\cdot,t)|^2
	+ \io |u_x(\cdot,t)-\whu_x(\cdot,t)|^2
	+ \io |\Theta(\cdot,t)-\whT(\cdot,t)|^2 \bigg\}
	+ \int_0^T \io |\Theta_x-\whTx|^2  \nn\\
	&\le& C(K,T) \cdot \bigg\{
	\io |u_{0t}-\whu_{0t}|^2
	+ \io |u_{0x}^2-\whu_{0x}|^2
	+ \io |\Theta_0-\whT_0|^2 
	\bigg\},
  \eea
  formally satisfied by two solutions $(u,\Theta)$ and $(\whu,\whT)$ of (\ref{0}) in $\Om\times (0,T)$ under the assumption that
  \be{03}
	\sup_{t\in (0,T)} \Big\{ \|\whu_t(\cdot,t)\|_{L^2(\Om)} + \|\Theta(\cdot,t)\|_{L^1(\Om)} \Big\}
	+ \int_0^T \io \Theta_x^2 \le K,
  \ee
  can adequately be exploited at rigorous levels.\abs
  In the context of a suitably designed testing procedure, Section \ref{sect3} addresses this for
  arbitrary weak solutions with additional
  regularity properties in the flavor of (\ref{reg}), thus asserting the claims on uniqueness and continuous
  dependence in Theorem \ref{theo99} and Proposition \ref{prop98}.
  Section \ref{sect4} will thereafter be devoted to the construction of solutions in this regularity class,
  where an estimate in the flavor of (\ref{02}) is used in two essential places:
  In a first step concentrating on initial data such that
  \be{04}
	u_0 \in C_0^\infty(\Om),
	\quad
	u_{0t} \in C_0^\infty(\Om),
	\quad 
	\quad \mbox{and} \quad
	\Theta_0 \in C^\infty(\bom)
	\mbox{ is positive in $\bom$ and such that } \Theta_{0x} \in C_0^\infty(\Om),
  \ee
  we shall see that, in line with (\ref{02}), solutions $(\veps,\ueps,\Teps)$ to a parabolic regularization of (\ref{0}),
  known to approximate $(u_t,u,\Theta)$ with a possibly irregular weak solution $(u,\Theta)$ of (\ref{0}) 
  in certain weak topologies (see (\ref{0eps}) and Lemma \ref{lem51}), actually satisfy an $\eps$-independent 
  inequality of the form
  \bea{05}
	& & \hs{-20mm}
	\sup_{t\in (0,T)} \io \big|\veps(\cdot,t+h)-\veps(\cdot,t)\big|^2
	+ \io \big|\uepsx(\cdot,t+h)-\uepsx(\cdot,t)\big|^2
	+ \io \big|\Teps(\cdot,t+h)-\Teps(\cdot,t)\big|^2 \nn\\
	&\le& C(T) \cdot \bigg\{
	\io \big| \veps(\cdot,h)-u_{0t} \big|^2
	+ \io \big| \uepsx(\cdot,h)-u_{0x}\big|^2
	+ \io \big| \Teps(\cdot,h)-\Theta_0\big|^2 \bigg\}
  \eea
  for $h\in (0,1)$ (Lemma \ref{lem10}).
  According to a key observation on a short-time higher-order regularity property valid for initial data fulfilling (\ref{04})
  (Lemma \ref{lem6}), the right-hand side herein can be seen to conveniently vanish in the limit $h\searrow 0$,
  ensuring that for any such data a global solution satisfying (\ref{reg}) indeed exists (Lemma \ref{lem111}).
  In Lemma \ref{lem11}, finally, general initial data merely complying with (\ref{init}) are covered by means
  of an independent second approximation argument, using the property in (\ref{02}) in deriving a Cauchy feature
  of a sequence of approximate solutions that evolve from regularized data.
\mysection{Two preliminary estimates for differences}\label{sect2}
Let us begin by deriving the following lemma which forms 
a common technical preliminary both for core parts of our analysis toward uniqueness and continuous dependence 
(see Lemma \ref{lem14} and Lemma \ref{lem21}) and for our construction of a global solution with the regularity properties
in (\ref{reg}) (cf.~Lemma \ref{lem10}).
\begin{lem}\label{lem13}
  Assume (\ref{f}), and let $\eta>0$ and $K>0$. 
  Then there exists $\Gamma_1(\eta,K)>0$ with the property that if 
  $v\in L^2(\Om), \whv\in L^2(\Om), \Theta\in W^{1,2}(\Om)$ and $\whT\in W^{1,2}(\Om)$ are such that
  $\Theta\ge 0$ and $\whT\ge 0$ as well as
  \be{13.01}
	\|\whv\|_{L^2(\Om)}
	+ \|\Theta\|_{L^1(\Om)}
	\le K,
  \ee
  we have
  \bea{13.1}
	& & \hs{-30mm}
	- \io \big\{ f'(\Theta)\Theta_x - f'(\whT)\whTx\big\} \cdot (v-\whv) \nn\\
	&\le& \eta \io (\Theta_x-\whTx)^2
	+ \Gamma_1(\eta,K) \cdot \bigg\{ 1 + \io \Theta_x^2 \bigg\} \cdot \bigg\{ \io (v-\whv)^2 + \io (\Theta-\whT)^2 \bigg\}
  \eea
  and
  \bea{13.2}
	& & \hs{-20mm}
	\io \big\{ f'(\Theta)\Theta_x v -f'(\whT) \whTx \whv\big\} \cdot (\Theta-\whT)
	+ \io \big\{ f(\Theta) v - f(\whT)\whv\big\} \cdot (\Theta_x -\whTx) \nn\\
	&\le& \eta \io (\Theta_x - \whTx)^2 
	+ \Gamma_1(\eta,K) \cdot \bigg\{ 1 + \io \Theta_x^2 \bigg\} \cdot \bigg\{ \io (v-\whv)^2 + \io (\Theta-\whT)^2 \bigg\}.
  \eea
\end{lem}
\proof
  According to a Gagliardo-Nirenberg inequality and the continuity of the embedding $W^{1,2}(\Om) \hra L^\infty(\Om)$,
  we can find $c_1>0$ and $c_2>0$ such that
  \be{13.3}
	\|\psi\|_{L^\infty(\Om)} \le c_1\|\psi_x\|_{L^2(\Om)}^\frac{1}{2} \|\psi\|_{L^2(\Om)}^\frac{1}{2}
	+ c_1\|\psi\|_{L^2(\Om)}
	\qquad \mbox{for all } \psi\in W^{1,2}(\Om)
  \ee
  and
  \be{13.31}
	\|\psi\|_{L^\infty(\Om)}^2 
	\le c_2 \|\psi_x\|_{L^2(\Om)}^2 + c_2\|\psi\|_{L^1(\Om)}^2
	\qquad \mbox{for all } \psi\in W^{1,2}(\Om),
  \ee
  and due to (\ref{f}) it is possible to choose $c_3>0$ and $c_4>0$ such that
  \be{13.32}
	0 \le f'(\xi)\le c_3
	\quad \mbox{and} \quad
	|f''(\xi)| \le c_4
	\qquad \mbox{for all } \xi\ge 0,
  \ee
  and that hence also
  \be{13.33}
	0 \le f(\xi) \le c_3 \xi
	\qquad \mbox{for all } \xi\ge 0.
  \ee
  Given 
  $(v,\Theta) \in L^2(\Om)\times W^{1,2}(\Om)$ and $(\whv,\whT)\in L^2(\Om)\times W^{1,2}(\Om)$, we decompose
  \be{13.4}
	\hs{-4mm}
	- \io \big\{ f'(\Theta)\Theta_x - f'(\whT)\whTx\big\} \cdot (v-\whv)
	= - \io f'(\whT) (\Theta_x-\whTx) (v-\whv)
	- \io \big\{ f'(\Theta)-f'(\whT)\big\} \Theta_x (v-\whv),
  \ee
  where 
  we use Young's inequality together with (\ref{13.32}) to see that
  \bea{13.5}
	- \io f'(\whT) (\Theta_x-\whTx) (v-\whv)
	&\le& \frac{\eta}{2} \io (\Theta_x-\whTx)^2
	+ \frac{1}{2\eta} \io f'^2(\whT) (v-\whv)^2 \nn\\
	&\le& \frac{\eta}{2} \io (\Theta_x-\whTx)^2
	+ \frac{c_3^2}{2\eta} \io (v-\whv)^2 
  \eea
  and, by the mean value theorem and (\ref{13.3}),
  \bea{13.6}
	& & \hs{-30mm}
	- \io \big\{ f'(\Theta)-f'(\whT)\big\} \Theta_x (v-\whv) \nn\\
	&\le& c_4 \io |\Theta-\whT| \cdot |\Theta_x| \cdot |v-\whv| \nn\\
	&\le& c_4 \|\Theta_x\|_{L^2(\Om)} \|\Theta-\whT\|_{L^\infty(\Om)} \|v-\whv\|_{L^2(\Om)} \nn\\
	&\le& c_1 c_4 \|\Theta_x\|_{L^2(\Om)} \|\Theta_x-\whTx\|_{L^2(\Om)}^\frac{1}{2} \|\Theta-\whT\|_{L^2(\Om)}^\frac{1}{2}
		\|v-\whv\|_{L^2(\Om)} \nn\\
	& & + c_1 c_4 \|\Theta_x\|_{L^2(\Om)} \|\Theta-\whT\|_{L^2(\Om)} \|v-\whv\|_{L^2(\Om)} \nn\\
	&=& \bigg\{ \frac{\eta}{2} \io (\Theta_x-\whTx)^2 \bigg\}^\frac{1}{4} 
		\cdot \Big(\frac{2}{\eta}\Big)^\frac{1}{4} c_1 c_4 \|\Theta_x\|_{L^2(\Om)} \|\Theta-\whT\|_{L^2(\Om)}^\frac{1}{2}
		\|v-\whv\|_{L^2(\Om)} \nn\\
	& & + c_1 c_4 \|\Theta_x\|_{L^2(\Om)} \|\Theta-\whT\|_{L^2(\Om)} \|v-\whv\|_{L^2(\Om)} \nn\\
	&\le& \frac{\eta}{2} \io (\Theta_x-\whTx)^2
	+ c_5(\eta) \|\Theta_x\|_{L^2(\Om)}^\frac{4}{3} \|\Theta-\whT\|_{L^2(\Om)}^\frac{2}{3} \|v-\whv\|_{L^2(\Om)}^\frac{4}{3} \nn\\
	& & + c_1 c_4 \|\Theta_x\|_{L^2(\Om)} \|\Theta-\whT\|_{L^2(\Om)} \|v-\whv\|_{L^2(\Om)} 
  \eea
  with $c_5(\eta):=(\frac{2}{\eta})^\frac{1}{3} \cdot (c_1 c_4)^\frac{4}{3}$.
  Since Young's inequality moreover ensures that
  \be{13.66}
	\|\Theta-\whT\|_{L^2(\Om)}^\iota \|v-\whv\|_{L^2(\Om)}^{2-\iota}
	\le \|\Theta-\whT\|_{L^2(\Om)}^2 + \|v-\whv\|_{L^2(\Om)}^2
	\qquad \mbox{for all } \iota\in (0,2)
  \ee
  as well as
  \be{13.67}
	\|\Theta_x\|_{L^2(\Om)}^\iota
	\le 1 + \|\Theta_x\|_{L^2(\Om)}^2 
	\qquad \mbox{for all } \iota\in (0,2),
  \ee
  from (\ref{13.4})-(\ref{13.6}) we obtain that
  \bea{13.7}
	& & \hs{-14mm}
	- \io \big\{ f'(\Theta)\Theta_x - f'(\whT)\whTx\big\} \cdot (v-\whv) \nn\\
	&\le& \Big(\frac{\eta}{2}+\frac{\eta}{2}\Big) \io (\Theta_x-\whTx)^2 
	+ \bigg\{ \frac{c_3^2}{2\eta} + c_5(\eta) + c_1 c_4 + (c_5(\eta)+c_1 c_4) \io \Theta_x^2 \bigg\} \cdot
		\io (v-\whv)^2 \nn\\
	& & 
	+ \bigg\{ c_5(\eta) + c_1 c_4 + (c_5(\eta)+c_1 c_4) \io \Theta_x^2 \bigg\} \cdot
		\io (\Theta-\whT)^2.
  \eea
  Next addressing (\ref{13.2}), we first rewrite
  \bea{13.8}
	& & \hs{-30mm}
	\io \big\{ f'(\Theta)\Theta_x v -f'(\whT) \whTx \whv\big\} \cdot (\Theta-\whT) \nn\\
	&=& \io f'(\Theta) \Theta_x (v-\whv) (\Theta-\whT)
	+ \io \big\{ f'(\Theta)-f'(\whT)\big\} \Theta_x \whv (\Theta-\whT) \nn\\
	& & + \io f'(\whT) (\Theta_x-\whTx) \whv (\Theta-\whT)
  \eea
  and again employ (\ref{13.32}), (\ref{13.3}), Young's inequality, (\ref{13.66}) and (\ref{13.67}) to estimate
  \bea{13.9}
	& & \hs{-18mm}
	\io f'(\Theta) \Theta_x (v-\whv)(\Theta-\whT) \nn\\
	&\le& c_3 \|\Theta_x\|_{L^2(\Om)} \|v-\whv\|_{L^2(\Om)} \|\Theta-\whT\|_{L^\infty(\Om)} \nn\\
	&\le& c_1 c_3 \|\Theta_x\|_{L^2(\Om)} \|v-\whv\|_{L^2(\Om)} 
		\|\Theta_x-\whTx\|_{L^2(\Om)}^\frac{1}{2} \|\Theta-\whT\|_{L^2(\Om)}^\frac{1}{2} \nn\\
	& & + c_1 c_3 \|\Theta_x\|_{L^2(\Om)} \|v-\whv\|_{L^2(\Om)} 
		\|\Theta-\whT\|_{L^2(\Om)} \nn\\
	&=& \bigg\{ \frac{\eta}{6} \io (\Theta_x-\whTx)^2 \bigg\}^\frac{1}{4} \cdot \Big(\frac{6}{\eta}\Big)^\frac{1}{4}
		c_1 c_3 \|\Theta_x\|_{L^2(\Om)} \|v-\whv\|_{L^2(\Om)} \|\Theta-\whT\|_{L^2(\Om)}^\frac{1}{2} \nn\\
	& & + c_1 c_3 \|\Theta_x\|_{L^2(\Om)} \|v-\whv\|_{L^2(\Om)} \|\Theta-\whT\|_{L^2(\Om)} \nn\\
	&\le& \frac{\eta}{6} \io (\Theta_x-\whTx)^2
	+ c_6(\eta) \|\Theta_x\|_{L^2(\Om)}^\frac{4}{3} \|v-\whv\|_{L^2(\Om)}^\frac{4}{3} \|\Theta-\whT\|_{L^2(\Om)}^\frac{2}{3} \nn\\
	& & + c_1 c_3 \|\Theta_x\|_{L^2(\Om)} \|v-\whv\|_{L^2(\Om)} \|\Theta-\whT\|_{L^2(\Om)} \nn\\
	&\le& \frac{\eta}{6} \io (\Theta_x-\whTx)^2
	+ (c_6(\eta)+c_1 c_3) 
	\cdot \bigg\{ 1 + \io \Theta_x^2 \bigg\} \cdot \bigg\{ \io (v-\whv)^2 + \io (\Theta-\whT)^2 \bigg\}
  \eea
  with $c_6(\eta):=(\frac{6}{\eta})^\frac{1}{3} \cdot (c_1 c_3)^\frac{4}{3}$.
  Relying on (\ref{13.01}) now, we moreover find that again due to (\ref{13.3}) and Young's inequality,
  \bea{13.10}
	& & \hs{-12mm}
	\io f'(\whT)(\Theta_x-\whTx) \whv (\Theta-\whT) \nn\\
	&\le& c_3 K \|\Theta_x-\whTx\|_{L^2(\Om)} \|\Theta-\whT\|_{L^\infty(\Om)} \nn\\
	&\le& c_1 c_3 K \|\Theta_x-\whTx\|_{L^2(\Om)}^\frac{3}{2} \|\Theta-\whT\|_{L^2(\Om)}^\frac{1}{2}
	+ c_1 c_3 K \|\Theta_x-\whTx\|_{L^2(\Om)} \|\Theta-\whT\|_{L^2(\Om)} \nn\\
	&=& \bigg\{ \frac{\eta}{12} \io (\Theta_x-\whTx)^2 \bigg\}^\frac{3}{4} 
		\cdot \Big(\frac{12}{\eta}\Big)^\frac{3}{4} c_1 c_3 K \|\Theta-\whT\|_{L^2(\Om)}^\frac{1}{2}
	+ c_1 c_3 K \|\Theta_x-\whTx\|_{L^2(\Om)} \|\Theta-\whT\|_{L^2(\Om)} \nn\\
	&\le& 
	\frac{\eta}{12} \io (\Theta_x-\whTx)^2
	+ \bigg\{ \Big(\frac{12}{\eta}\Big)^\frac{3}{4} c_1 c_3 K \|\Theta-\whT\|_{L^2(\Om)}^\frac{1}{2} \bigg\}^4 
	\nn\\
	& & 
	+ \frac{\eta}{12} \|\Theta_x-\whTx\|_{L^2(\Om)}^2
	+ \frac{(c_1 c_3 K)^2}{4\cdot \frac{\eta}{12}} \cdot \|\Theta-\whT\|_{L^2(\Om)}^2 
	\nn\\
	&=& \frac{\eta}{6} \io (\Theta_x-\whTx)^2
	+ c_7(\eta,K) \io (\Theta-\whT)^2
  \eea
  with $c_7(\eta,K):=(\frac{12}{\eta})^3 \cdot (c_1 c_3 K)^4 + \frac{3c_1^2 c_3^2 K^2}{\eta}$,
  while once more by the mean value theorem, (\ref{13.3}), Young's inequality and (\ref{13.67}),
  \bas
	& & \hs{-12mm}
	\io \big\{ f'(\Theta)-f'(\whT)\big\} \Theta_x \whv (\Theta-\whT) \nn\\
	&\le& c_4 \io |\Theta_x| \cdot |\whv| \cdot |\Theta-\whT|^2 \\
	&\le& c_4 K \|\Theta_x\|_{L^2(\Om)} \|\Theta-\whT\|_{L^\infty(\Om)}^2 \\
	&\le& 2c_1^2 c_4 K \|\Theta_x\|_{L^2(\Om)} \|\Theta_x-\whTx\|_{L^2(\Om)} \|\Theta-\whT\|_{L^2(\Om)}
	+ 2c_1^2 c_4 K \|\Theta_x\|_{L^2(\Om)} \|\Theta-\whT\|_{L^2(\Om)}^2 \\
	&\le& \frac{\eta}{6} \io (\Theta_x-\whTx)^2
	+ \frac{6c_1^4 c_4^2 K^2}{\eta} \cdot \bigg\{ \io \Theta_x^2 \bigg\} \cdot \io (\Theta-\whT)^2 \\
	& & + 2c_1^2 c_4 K \|\Theta_x\|_{L^2(\Om)} \|\Theta-\whT\|_{L^2(\Om)}^2 \\
	&\le& \frac{\eta}{6} \io (\Theta_x-\whTx)^2
	+ c_8(\eta,K) \cdot \bigg\{ 1 + \io \Theta_x^2 \bigg\} \cdot \io (\Theta-\whT)^2
  \eas
  if we let $c_8(\eta,K):=\frac{6c_1^2 c_4^2 K^2}{\eta} + 2c_1^2 c_4 K$, for instance.
  Therefore, (\ref{13.8})-(\ref{13.10}) ensure that
  \bea{13.11}
	& & \hs{-20mm}
	\io \big\{ f'(\Theta)\Theta_x v - f'(\whT)\whTx \whv\big\} \cdot (\Theta-\whT) \nn\\
	&\le& \!\!\! \Big(\frac{\eta}{6}+\frac{\eta}{6}+\frac{\eta}{6}\Big) \io (\Theta_x-\whTx)^2 \nn\\
	& & \!\!\! + \big\{ c_6(\eta)+ c_1 c_3 + c_7(\eta,K)+c_8(\eta,K)\big\} \cdot \bigg\{ 1 + \io \Theta_x^2 \bigg\} 
		\cdot 
	\bigg\{ \io (\Theta-\whT)^2
	+ \io (v-\whv)^2 \bigg\},
  \eea
  so that it remains to suitably estimate
  \bea{13.12}
	\hs{-4mm}
	\io \big\{ f(\Theta) v - f(\whT)\whv\big\} \cdot (\Theta_x-\whTx)
	= \io \big\{ f(\Theta)-f(\whT)\big\} \whv (\Theta_x-\whTx) 
	+ \io f(\Theta)(v-\whv)(\Theta_x-\whTx).
  \eea
  Indeed, the mean value theorem, (\ref{13.01}), (\ref{13.3}) and Young's inequality guarantee that
  \bas
	& & \hs{-24mm}
	\io \big\{ f(\Theta)-f(\whT)\big\} \whv (\Theta_x-\whTx) \nn\\
	&\le& c_3 \io |\Theta-\whT| \cdot |\whv| \cdot |\Theta_x-\whTx| \\
	&\le& c_3 K \|\Theta-\whT\|_{L^\infty(\Om)} \|\Theta_x-\whTx\|_{L^2(\Om)} \\
	&\le& c_1 c_3 K \|\Theta_x-\whTx\|_{L^2(\Om)}^\frac{3}{2} \|\Theta-\whT\|_{L^2(\Om)}^\frac{1}{2}
	+ c_1 c_3 K \|\Theta-\whT\|_{L^2(\Om)} \|\Theta_x-\whTx\|_{L^2(\Om)} \\
	&=& \bigg\{ \frac{\eta}{8} \io (\Theta_x-\whTx)^2 \bigg\}^\frac{3}{4} \cdot 
		\Big(\frac{8}{\eta}\Big)^\frac{3}{4} c_1 c_3 K \|\Theta-\whT\|_{L^2(\Om)}^\frac{1}{2} \\
	& & + c_1 c_3 K \|\Theta-\whT\|_{L^2(\Om)} \|\Theta_x-\whTx\|_{L^2(\Om)} \\
	&\le&
	\frac{\eta}{8} \io (\Theta_x-\whTx)^2 
	+ \bigg\{ \Big(\frac{8}{\eta}\Big)^\frac{3}{4} c_1 c_3 K \|\Theta-\whT\|_{L^2(\Om)}^\frac{1}{2} \bigg\}^4 \\
	& &
	+ \frac{\eta}{8} \|\Theta_x-\whTx\|_{L^2(\Om)}^2
	+ \frac{(c_1 c_3 K)^2}{4\cdot\frac{\eta}{8}} \cdot \|\Theta-\whT\|_{L^2(\Om)}^2 \\
	&=& \frac{\eta}{4} \io (\Theta_x-\whTx)^2
	+ c_9(\eta,K) \io (\Theta-\whT)^2
  \eas
  with $c_9(\eta,K):=(\frac{8}{\eta})^3 \cdot (c_1 c_3 K)^4 + \frac{2c_1^2 c_3^2 K^2}{\eta}$, whereas
  thanks to the mean value theorem, (\ref{13.31}) and (\ref{13.01}),
  \bas
	\io f(\Theta) (v-\whv) (\Theta_x-\whTx)
	&\le& c_3 \io \Theta |v-\whv| \cdot |\Theta_x-\whTx| \\
	&\le& c_3 \|\Theta\|_{L^\infty(\Om)} \|v-\whv\|_{L^2(\Om)} \|\Theta_x-\whTx\|_{L^2(\Om)} \\
	&\le& \frac{\eta}{4} \io (\Theta_x-\whTx)^2
	+ \frac{c_3^2}{\eta} \|\Theta\|_{L^\infty(\Om)}^2 \io (v-\whv)^2 \\
	&\le& \frac{\eta}{4} \io (\Theta_x-\whTx)^2
	+ \frac{c_2 c_3^2}{\eta} \cdot \bigg\{ K^2 + \io \Theta_x^2 \bigg\} \cdot \io (v-\whv)^2.
  \eas
  Accordingly, (\ref{13.12}) implies that
  \bas
	& & \hs{-24mm}
	\io \big\{ f(\Theta) v - f(\whT)\whv\big\} \cdot (\Theta_x-\whTx) \\
	&\le& \Big(\frac{\eta}{4}+\frac{\eta}{4}\Big) \cdot \io (\Theta_x-\whTx)^2 \\
	& & + \bigg\{ c_9(\eta,K) + \frac{c_2 c_3^2 K^2}{\eta} + \frac{c_2 c_3^2}{\eta} \io \Theta_x^2 \bigg\} \cdot \bigg\{
		\io (v-\whv)^2 + \io (\Theta-\whT)^2 \bigg\},
  \eas
  and hence, after having been added to (\ref{13.11}) and combined with (\ref{13.7}), shows that both (\ref{13.1}) and
  (\ref{13.2}) hold if $\Gamma_1(\eta,K)$ is appropriately large.
\qed
\mysection{Uniqueness and continuous dependence}\label{sect3}
This section will mainly be devoted to those parts of our main results which are related to uniqueness of solutions and their
continuous dependence on the initial data. By providing statements which go slightly beyond the respective
formulations in Theorem \ref{theo99} and Proposition \ref{prop98}, however, we can moreover accomplish a key step
toward our proof of existence in the next section (see Lemma \ref{lem97}). 
For this purpose, it will be technically convenient to here examine solutions consistent with the requirements
in (\ref{reg}) but defined on time intervals of finite length only, thus assuming that
\be{99.1}
	\lbal
	u\in C^0([0,T];W_0^{1,2}(\Om)), \\[1mm]
	u_t \in C^0([0,T];L^2(\Om))
	\qquad \mbox{and} \\[1mm]
	\Theta\in C^0([0,T];L^2(\Om)) \cap L^2((0,T);W^{1,2}(\Om)).
	\ear
\ee
Our first step in the analysis of two different solutions complying with these hypothesis concentrates on the 
parabolic subsystem of (\ref{0}), and utilizes a suitably arranged testing procedure therefor to estimate
differences as follows.
\begin{lem}\label{lem14}
  Assume (\ref{f}), and let $K>0$.
  Then there exists $\Gamma_2(K)>0$ such that if $(u_0,u_{0t},\Theta_0)$ and $(\whu_0,\whu_{0t},\whT_0)$
  satisfy (\ref{init}), and if $T>0$ and $(u,\Theta)$ as well as $(\whu,\whT)$ are weak solutions of (\ref{0}) in
  $\Om\times (0,T)$ with initial data $(u_0,u_{0t},\Theta_0)$ and $(\whu_0,\whu_{0t},\whT_0)$, respectively,
  which both have the regularity properties in (\ref{99.1}) and moreover satisfy
  \be{14.01}
	\|\whu_t(\cdot,t)\|_{L^2(\Om)}
	+ \|\Theta(\cdot,t)\|_{L^1(\Om)}
	\le K
	\qquad \mbox{for all } t\in (0,T),
  \ee
  then 
  \bea{14.2}
	& & \hs{-20mm}
	\io |\Theta(\cdot,t)-\whT(\cdot,t)|^2
	+ \int_0^t \io (\Theta_x-\whTx)^2 \nn\\
	&\le& \io |\Theta_0-\whT_0|^2 \nn\\
	& & + \Gamma_2(K) \cdot \int_0^t \bigg\{ 1 + \io \Theta_x^2(\cdot,s) \bigg\} \cdot \bigg\{
		\io | u_t(\cdot,s)-\whu_t(\cdot,s)|^2
		+ \io |\Theta(\cdot,s)-\whT(\cdot,s)|^2 \bigg\} ds
  \eea
  for all $t\in (0,T)$.
\end{lem}
\proof
  For convenience in notation, we let $(u,\Theta)(x,t):=(0,0)$ and $(\whu,\whT)(x,t):=(0,0)$ for $(x,t)\in \Om\times (T,\infty)$,
  and write $d:=\Theta-\whT$ as well as $d_0:=\Theta_0-\whT_0$.
  Twice applying (\ref{wt}), we obtain that for each $\vp\in C_0^\infty(\bom\times [0,\infty))$
  with $\supp\vp\subset \bom\times [0,T)$,
  \be{14.3}
	- \int_0^\infty \io d\vp_t 
	- \io d_0\vp(\cdot,0)
	= - \int_0^\infty \io d_x \vp_x
	+ \int_0^\infty \io g_1 \vp
	+ \int_0^\infty \io g_2 \vp_x,
  \ee
  where
  \be{14.4}
	g_1(x,t):=\lball
	\big\{ f'(\Theta) \Theta_x u_t - f'(\whT)\whTx\whu_t \big\} (x,t),
	\qquad & (x,t)\in \Om \times (0,T), \\[1mm]
	0,
	& (x,t)\in\Om\times (T,\infty),
	\ear
  \ee
  and
  \be{14.5}
	g_2(x,t):=\lball
	\big\{ f(\Theta) u_t - f(\whT)\whu_t \big\} (x,t),
	\qquad & (x,t)\in \Om \times (0,T), \\[1mm]
	0,
	& (x,t)\in\Om\times (T,\infty).
	\ear
  \ee
  Here we note that according to (\ref{99.1}) and (\ref{f}) we have
  \be{14.41}
	\{ d,d_x,g_2\} \subset L^2(\Om\times (0,\infty))
  \ee
  and
  \be{14.42}
	g_1 \in L^2((0,\infty);L^1(\Om)),
  \ee
  whence by a straightforward approximation argument it follows that (\ref{14.3}) extends so as to remain valid actually for any
  $\vp\in L^2((0,\infty);W^{1,2}(\Om)) \hra L^2((0,\infty);L^\infty(\Om))$ 
  which additionally satisfies $\vp_t\in L^2(\Om\times (0,\infty))$ and
  $\vp=0$ a.e.~in $\Om\times (T_0,\infty)$ with some $T_0\in (0,T)$.
  For fixed $t_0\in (0,T), \del\in (0,T-t_0)$ and $h\in (0,T-t_0-\del)$, we may therefore apply (\ref{14.3}) to
  \bas
	\vp(x,t)\equiv \vp_{\del, h}(x,t) := \zd(t) \cdot (S_h d)(x,t),
	\qquad (x,t)\in\Om\times (0,\infty),
  \eas
  where $\zd$ is the piecewise linear function on $\R$ satisfying $\zd\equiv 1$ on $(-\infty,t_0]$,
  $\zd\equiv 0$ on $[t_0+\del,\infty)$ and $\zd'\equiv -\frac{1}{\del}$ on $(t_0,t_0+\del)$, and where the
  Steklov averages
  $S_h \phi$ of $\phi\in L^1(\Om\times (0,\infty))$ are defined in a classical manner by letting
  \be{14.55}
	(S_h \phi)(x,t):=\frac{1}{h} \int_t^{t+h} \phi(x,s) ds,
	\qquad (x,s)\in\Om\times (0,\infty).
  \ee
  We thereby obtain the identity
  \bea{14.6}
	& & \hs{-30mm}
	\frac{1}{\del} \int_{t_0}^{t_0+\del} \io d(x,t) (S_h d)(x,t) dxdt
	- \int_0^\infty \io \zd(t) d(x,t) \cdot \frac{d(x,t+h)-d(x,t)}{h} dxdt \nn\\
	& & \hs{-20mm}
	- \io d_0(x) (S_h d)(x,0) dx \nn\\
	&=& - \int_0^\infty \io \zd(t) d_x(x,t) (S_h d_x)(x,t) dxdt \nn\\
	& & + \int_0^\infty \io \zd(t) g_1(x,t) (S_h d)(x,t) dxdt \nn\\
	& & + \int_0^\infty \io \zd(t) g_2(x,t) (S_h d_x)(x,t) dxdt
  \eea
  for all $\del\in (0,T-t_0)$ and $h\in (0,T-t_0-\del)$, 
  where we observe that from (\ref{14.41}) and standard stabilization properties of
  Steklov averaging operators (\cite[Chapter 3-(i)]{dibenedetto})
  it readily follows that
  \be{14.7}
	S_h d_x \to d_x
	\ \mbox{in } L^2(\Om\times (0,\infty))
	\qquad \mbox{as } h\searrow 0,
  \ee
  and that, again since $L^2((0,\infty);W^{1,2}(\Om)) \hra L^2((0,\infty);L^\infty(\Om))$, also
  \be{14.8}
	S_h d \to d
	\ \mbox{in } L^2((0,\infty);L^\infty(\Om))
	\qquad \mbox{as } h\searrow 0,
  \ee
  while directly from (\ref{14.55}) and the fact that 
  \be{14.80}
	d(\cdot,t) \to d_0
	\ \mbox{in } L^2(\Om)
	\qquad \mbox{as } t\searrow 0,
  \ee
  as clearly asserted by (\ref{99.1}),
  we obtain that
  \be{14.81}
	(S_h d)(\cdot,0) \to d_0
	\ \mbox{in } L^2(\Om)
	\qquad \mbox{as } h\searrow 0.
  \ee
  Since furthermore, by Young's inequality, a linear substitution and the fact that $\zd\equiv 1$ on $(-\infty,0)$ for 
  all $\del\in (0,T-t_0)$,
  \bas
	& & \hs{-16mm}
	- \int_0^\infty \io \zd(t) d(x,t) \cdot \frac{d(x,t+h)-d(x,t)}{h} dxdt \\
	&=& - \frac{1}{h} \int_0^\infty \io \zd(t) d(x,t) d(x,t+h) dxdt
	+ \frac{1}{h} \int_0^\infty \io \zd(t) d^2(x,t) dxdt \\
	&\ge& - \frac{1}{2h} \int_0^\infty \io \zd(t) d^2(x,t+h) dxdt
	+ \frac{1}{2h} \int_0^\infty \io \zd(t) d^2(x,t) dxdt \\
	&=& - \frac{1}{2h} \int_h^\infty \io \zd(t'-h) d^2(x,t') dxdt'
	+ \frac{1}{2h} \int_0^\infty \io \zd(t) d^2(x,t) dxdt \\
	&=&  \frac{1}{2h} \int_0^h \io d^2(x,t') dxdt' \nn\\
	& & - \frac{1}{2h} \int_0^\infty \io \zd(t'-h) d^2(x,t') dxdt'
	+ \frac{1}{2h} \int_0^\infty \io \zd(t) d^2(x,t) dxdt \\
	&=&  \frac{1}{2h} \int_0^h \io d^2(x,t') dxdt' \\
	& & + \frac{1}{2} \int_0^\infty \io \frac{\zd(t)-\zd(t-h)}{h} d^2(x,t) dxdt
	\qquad \mbox{for all $\del\in (0,T-t_0)$ and } h\in (0,T-t_0-\del),
  \eas
  and since $\frac{\zd(\cdot)-\zd(\cdot-h)}{h} \wsto \zd'$ in $L^\infty((0,\infty))$ as $h\searrow 0$, combining (\ref{14.7}),
  (\ref{14.8}) and (\ref{14.81}) with, again, (\ref{14.80}), (\ref{14.41}) and (\ref{14.42}), 
  from (\ref{14.6}) we infer on letting $h\searrow 0$ that
  \bas
	\frac{1}{2\del} \int_{t_0}^{t_0+\del} \io d^2
	- \frac{1}{2} \io d_0^2
	&=& \frac{1}{\del} \int_{t_0}^{t_0+\del} \io d^2
	+ \bigg\{ \frac{1}{2} \io d_0^2 - 
	\frac{1}{2\del} \int_{t_0}^{t_0+\del} \io d^2
	\bigg\}
	- \io d_0^2 \\
	&\le& - \int_0^\infty \io \zd(t) d_x^2(x,t) dxdt
	+ \int_0^\infty \io \zd(t) g_1(x,t) d(x,t) dxdt \\
	& & + \int_0^\infty \io \zd(t) g_2(x,t) d_x(x,t) dxdt
	\qquad \mbox{for all } \del\in (0,T-t_0).
  \eas
  As $0\le t\mapsto \io d^2(\cdot,t)$ is continuous at $t=t_0$ by (\ref{99.1}), again in view of (\ref{14.41}) and (\ref{14.42})
  we may let $\del\searrow 0$ here to see that thus
  \be{14.9}
	\frac{1}{2} \io d^2(\cdot,t_0)
	+ \int_0^{t_0} \io d_x^2
	\le 
	\frac{1}{2} \io d_0^2
	+ \int_0^{t_0} \io g_1 d
	+ \int_0^{t_0} \io g_2 d_x.
  \ee
  Here, on the basis of (\ref{14.01})
  we may rely on Lemma \ref{lem13} to see that if we let $c_1\equiv c_1(K):=\Gamma_1(\frac{1}{2},K)$ with $\Gamma_1(\cdot,\cdot)$
  as found there, then
  \bas
	& & \hs{-20mm}
	\int_0^{t_0} \io g_1 d
	+ \int_0^{t_0} \io g_2 d_x \\
	&=& \int_0^{t_0} \io \big\{ f'(\Theta) \Theta_x u_t - f'(\whT) \whTx\whu_t\big\} \cdot (\Theta-\whT) 
	+ \int_0^{t_0} \io \big\{ f(\Theta)u_t - f(\whT) \whu_t\big\} \cdot (\Theta_x-\whTx) \\
	&\le& \frac{1}{2} \int_0^{t_0} \io (\Theta_x-\whTx)^2 \\
	& & + c_1 \int_0^{t_0} \bigg\{ 1 + \io \Theta_x^2(\cdot,s) \bigg\} \cdot 
		\bigg\{ \io |u_t(\cdot,s)-\whu_t(\cdot,s)|^2
		+ \io |\Theta(\cdot,s)-\whT(\cdot,s)|^2 \bigg\} ds.
  \eas
  As $t_0\in (0,T)$ was arbitrary, together with (\ref{14.9}) this yields (\ref{14.2}) with $\Gamma_2(K):=2c_1$.
\qed
Our derivation of corresponding features in the hyperbolic part of (\ref{0}) will use the following 
basic observation on validity of energy identities along weak trajectories in forced linear wave equations,
as recorded in \cite[Lemma 5.2]{win_existence}.
\begin{lem}\label{lem877}
  Let $T>0$, $g\in L^1((0,T);L^2(\Om))$, $u_0\in W_0^{1,2}(\Om)$ and $u_{0t}\in L^2(\Om)$, and suppose that
  \be{8.1}
	u\in C^0([0,T);L^2(\Om)) \cap L^\infty((0,T);W_0^{1,2}(\Om))
	\qquad \mbox{is such that} \qquad
	u_t \in L^\infty((0,T);L^2(\Om)),
  \ee
  and that for all $\vp\in C_0^\infty(\Om\times [0,T))$,
  \be{8.3}
	\int_0^T \io u\vp_{tt} - \io u_{0t} \vp(\cdot,0) + \io u_0 \vp_t(\cdot,0)
	= - \int_0^T \io u_x \vp_x
	+ \int_0^T \io g\vp.
  \ee
  Then 
  \be{877.4}
	\frac{1}{2} \io u_t^2(\cdot,t) 
	+ \frac{1}{2} \io u_x^2(\cdot,t)
	= \frac{1}{2} \io u_{0t}^2
	+ \frac{1}{2} \io u_{0x}^2
	+ \int_0^t \io g u_t
	\qquad \mbox{for a.e.~} t\in (0,T).
  \ee
\end{lem}
An application to differences of two solutions to (\ref{0}) yields a mechanical counterpart of Lemma \ref{lem14},
and by combination with the latter hence leads to the following.
\begin{lem}\label{lem21}
  Assume (\ref{f}), and let $K>0$ and $T>0$.
  Then there exists $\Gamma_3(K,T)>0$ with the property that if $(u_0,u_{0t},\Theta_0)$ and $(\whu_0,\whu_{0t},\whT_0)$ 
  satisfy (\ref{init}),
  and if $(u,\Theta)$ and $(\whu,\whT)$ are weak solutions of (\ref{0}) in $\Om\times (0,T)$ emanating from
  $(u_0,u_{0t},\Theta_0)$ and $(\whu_0,\whu_{0t},\whT_0)$, respectively, and additionally fulfilling (\ref{99.1})
  as well as
  \be{21.1}
	\|\whu_t(\cdot,t)\|_{L^2(\Om)} 
	+ \|\Theta(\cdot,t)\|_{L^1(\Om)}
	\le K
	\qquad \mbox{for all } t\in (0,T)
  \ee
  and
  \be{21.2}
	\int_0^T \io \Theta_x^2 \le K,
  \ee
  then
  \bea{21.3}
	& & \hs{-16mm}
	\sup_{t\in (0,T)} \bigg\{
	\io |u_t(\cdot,t)-\whu_t(\cdot,t)|^2
	+ \io |u_x(\cdot,t)-\whu_x(\cdot,t)|^2
	+ \io |\Theta(\cdot,t)-\whT(\cdot,t)|^2 \bigg\}
	+ \int_0^T \io |\Theta_x-\whTx|^2  \nn\\
	&\le& \Gamma_3(K,T) \cdot \bigg\{
	\io |u_{0t}-\whu_{0t}|^2
	+ \io |u_{0x}^2-\whu_{0x}|^2
	+ \io |\Theta_0-\whT_0|^2 
	\bigg\}.
  \eea
\end{lem}
\proof
  Subtracting the respective versions of (\ref{wu}) satisfied by $(u,\Theta)$ and $(\whu,\whT)$, we see that
  for all $\vp\in C_0^\infty(\Om\times [0,T))$,
  \bas
	\int_0^T \io (u-\whu) \vp_{tt}
	- \io (u_{0t}-\whu_{0t}) \vp(\cdot,0)
	+ \io (u_{0x}-\whu_{0x}) \vp_t(\cdot,0)
	= - \int_0^T \io (u_x-\whux) \vp_x
	+ \int_0^T \io g\vp,
  \eas
  where $g:=-f'(\Theta) \Theta_x - f'(\whT) \whTx$.
  Since (\ref{99.1}) particularly ensures that thanks to (\ref{f}) we have $g\in L^1((0,T);L^2(\Om))$, and that clearly also
  $u-\whu\in C^0([0,T];L^2(\Om)) \cap L^\infty((0,T);W_0^{1,2}(\Om))$ and
  $u_t-\whu_t \in L^\infty((0,T);L^2(\Om))$, we may thus employ Lemma \ref{lem877} to see that
  for all $t\in (0,T)$,
  \bea{99.2}
	& & \hs{-20mm}
	\frac{1}{2} \io |u_t(\cdot,t)-\whu_t(\cdot,t)|^2
	+ \frac{1}{2} \io |u_x(\cdot,t)-\whux(\cdot,t)|^2 \nn\\
	&=& 
	\frac{1}{2} \io |u_{0t}-\whu_{0t}|^2
	+ \frac{1}{2} \io |u_{0x}-\whu_{0x}|^2
	+ \int_0^t \io g\cdot (u_t-\whu_t) \nn\\
	&=&
	\frac{1}{2} \io |u_{0t}-\whu_{0t}|^2
	+ \frac{1}{2} \io |u_{0x}-\whu_{0x}|^2
	- \int_0^t \io \big\{ f'(\Theta) \Theta_x - f'(\whT)\whTx \big\} \cdot (u_t-\whu_t),
  \eea
  because (\ref{99.1}) ensures continuity of the two integrals on the left-hand side herein with respect to the variable $t\in (0,T)$,
  as well as of $(0,T) \ni t\mapsto \int_0^t \io g\cdot (u_t-\whu_t)$.
  Now in view of (\ref{21.1}) we may invoke Lemma \ref{lem13} to see that if we let $c_1\equiv c_1(K):=\Gamma_1(\frac{1}{2},K)$
  with $\Gamma_1(\cdot,\cdot)$ as introduced there, then
  for all $t\in (0,T)$,
  \bea{99.3}
	& & \hs{-20mm}
	- \int_0^t \io \big\{ f'(\Theta) \Theta_x - f'(\whT)\whTx \big\} \cdot (u_t-\whu_t) \nn\\
	&\le& \frac{1}{2} \int_0^t \io (\Theta_x-\whTx)^2 \nn\\
	& & + c_1\cdot\int_0^t \bigg\{ 1 + \io \Theta_x^2(\cdot,s)\bigg\} \cdot \bigg\{
		\io |u_t(\cdot,s)-\whu_t(\cdot,s)|^2
		+ \io |\Theta(\cdot,s)-\whT(\cdot,s)|^2 \bigg\} ds,
  \eea
  while again thanks to (\ref{21.1}), Lemma \ref{lem14} shows that if we let $\Gamma_2(K)$ be as given there, then
  \bea{99.4}
	& & \hs{-20mm}
	\io |\Theta(\cdot,t)-\whT(\cdot,t)|^2
	+ \int_0^t \io (\Theta_x-\whTx)^2 \nn\\
	&\le& \io |\Theta_0-\whT_0|^2  \nn\\
	& & \hs{-4mm}
	+ \Gamma_2(K) \cdot\int_0^t \bigg\{ 1 + \io \Theta_x^2(\cdot,s)\bigg\} \cdot \bigg\{
		\io |u_t(\cdot,s)-\whu_t(\cdot,s)|^2
		+ \io |\Theta(\cdot,s)-\whT(\cdot,s)|^2 \bigg\} ds
  \eea
  for all $t\in (0,T)$.
  In combination, (\ref{99.2})-(\ref{99.4}) show that
  \bas
	y(t):=\frac{1}{2} \io |u_t(\cdot,t)-\whu_t(\cdot,t)|^2
	+ \frac{1}{2} \io |u_x(\cdot,t)-\whux(\cdot,t)|^2
	+ \io |\Theta(\cdot,t)-\whT(\cdot,t)|^2,
	\qquad t\in [0,T],
  \eas
  and
  \bas
	y_0:=\frac{1}{2} \io |u_{0t}-\whu_{0t}|^2
	+ \frac{1}{2} \io |u_{0x}-\whu_{0x}|^2
	+ \io |\Theta_0-\whT_0|^2
  \eas
  as well as
  \bas
	b(t):=(c_1+\Gamma_2(K)) \cdot \bigg\{ 1 + \io \Theta_x^2(\cdot,t)\bigg\}, 
	\qquad t\in (0,T),
  \eas
  satisfy
  \be{99.5}
	y(t) 
	+ \frac{1}{2} \int_0^t \io (\Theta_x-\whT_x)^2
	\le y_0 + \int_0^t b(s) y(s) ds
	\qquad \mbox{for all } t\in [0,T].
  \ee
  Since (\ref{99.1}) ensures that $y\in C^0([0,T])$ and $b\in L^1((0,T))$, 
  and since our assumption in (\ref{21.2}) warrants that
  \bas
	\int_0^T b(t) dt 
	\le c_2\equiv c_2(K,T):=(c_1+\Gamma_2(K)) \cdot (T+K),
  \eas
  a Gr\"onwall lemma asserts that (\ref{99.5}), firstly, implies the inequality
  \bas
	y(t) \le y_0 e^{\int_0^t b(s) ds} \le e^{c_2} y_0
	\qquad \mbox{for all } t\in (0,T),
  \eas
  and thus, secondly, that
  \bas
	\frac{1}{2} \int_0^T \io (\Theta_x-\whT_x)^2
	\le y_0 + c_2 e^{c_2} y_0.
  \eas
  An evident choice of $\Gamma_3(K,T)$ therefore leads to (\ref{21.3}).
\qed
The announced uniqueness statement, actually in a slightly more general version including finite time intervals,
thus becomes an immediate consequence.
\begin{prop}\label{prop22}
  If (\ref{f}) and (\ref{init}) hold, 
  then for each $T>0$, the problem (\ref{0}) admits at most one weak solution $(u,\Theta)$ 
  in $\Om\times (0,T)$ which has the additional property that (\ref{99.1}) holds.
\end{prop}
\proof
  If $(u,\Theta)$ and $(\whu,\whT)$ are two weak solutions fulfilling (\ref{99.1}), then (\ref{21.1}) and (\ref{21.2})
  are clearly satisfied with some $K>0$ that may possibly depend on $(u,\Theta)$ and $(\whu,\whT)$.
  As both these solutions emanate from the same initial data, the claim hence immediately results from Lemma \ref{lem21}.
\qed
In order to prepare a second application of Lemma \ref{lem21} with the intention to provide 
a quantitative estimate for differences of solutions,
to be used in the proof not only of Proposition \ref{prop98} but also of the existence claim in Theorem \ref{theo99},
we again rely on Lemma \ref{lem877} and combine its result with a fairly simple observation on mass evolution in (\ref{wt})
to confirm the following rigorous counterpart of the energy identity in (\ref{energy}).
\begin{lem}\label{lem23}
  Assume (\ref{f}) and (\ref{init}), and let $T>0$ and $(u,\Theta)$ be a weak solution of (\ref{0}) in $\Om\times (0,T)$
  which satisfies (\ref{99.1}). Then
  \be{23.1}
	\frac{1}{2} \io u_t^2(\cdot,t)
	+ \frac{1}{2} \io u_x^2(\cdot,t)
	+ \io \Theta(\cdot,t)
	= \frac{1}{2} \io u_{0t}^2
	+ \frac{1}{2} \io u_{0x}^2
	+ \io \Theta_0
	\qquad \mbox{for all } t\in (0,T).
  \ee
\end{lem}
\proof
  Given $t_0\in (0,T)$, for $\del\in (0,T-t_0)$ we again let $\zd\in W^{1,\infty}(\R)$ be such that $\zd\equiv 1$ in $(-\infty,t_0]$,
  $\zd\equiv 0$ in $[t_0+\del,\infty)$ and $\zd'\equiv -\frac{1}{\del}$ in $(t_0,t_0+\del)$.
  Once more by straightforward approximation thereof by smooth functions, we see that (\ref{wt}) in fact as well holds for
  $\vp(x,t)\equiv \vp_\del(x,t):=\zd(t)$, $(x,t)\in\bom\times [0,T)$, and that thus
  \bas
	\frac{1}{\del} \int_{t_0}^{t_0+\del} \io \Theta(x,t) dxdt
	- \io \Theta_0(x) dx
	= \int_0^T \io \zd(t) f'(\Theta(x,t))\Theta_x(x,t) u_t(x,t) dxdt
  \eas
  By continuity of $[0,T)\ni t\mapsto \io \Theta(x,t) dx$ at $t=t_0$, in the limit $\del\searrow 0$ this implies that
  for any such $t_0$ we have
  \bas
	\io \Theta(\cdot,t_0) = \io \Theta_0 + \int_0^{t_0} \io f'(\Theta) \Theta_x u_t.
  \eas
  When combined with the outcome of Lemma \ref{lem877}, this confirms (\ref{23.1}).
\qed
Together with Lemma \ref{lem21}, this yields bounds for fairly arbitrary weak solutions of (\ref{0}) fulfilling (\ref{99.1}),
among which we extract the ones needed below.
\begin{lem}\label{lem24}
  Assume (\ref{f}), and let $K>0$ and $T>0$.
  Then there exists $\Gamma_4(K,T)>0$ such that if (\ref{init}) holds with
  \be{i1}
	\|v_0\|_{L^2(\Om)}
	+ \|u_0\|_{W^{1,2}(\Om)}
	+ \|\Theta_0\|_{L^2(\Om)} 
	\le K,
  \ee
  and if $(u,\Theta)$ is a weak solution
  of (\ref{0}) in $\Om\times (0,T)$ which satisfies (\ref{99.1}), then
  \be{24.1}
	\sup_{t\in (0,T)} \big\{ \|u_t(\cdot,t)\|_{L^2(\Om)} + \|\Theta(\cdot,t)\|_{L^1(\Om)} \big\}
	+ \int_0^T \io \Theta_x^2
	\le \Gamma_4(K,T).
  \ee
\end{lem}
\proof
  We first observe that (\ref{i1}) particularly controls the expression on the right of (\ref{23.1}), so that 
  using Lemma \ref{lem23} we obtain $c_1=c_1(K)>0$ such that any weak solution of (\ref{0}) fulfilling
  (\ref{99.1}) satisfies
  \be{24.2}
	\|u_t(\cdot,t)\|_{L^2(\Om)} \le c_1
	\qquad \mbox{for all } t\in (0,T).
  \ee
  As $(0,0)$ trivially is another weak solution of (\ref{0}) with this additional property, this especially enables us to
  infer from Lemma \ref{lem21} that with $\Gamma_3(\cdot,\cdot)$ obtained there,
  \bas
	\sup_{t\in (0,T)} \bigg\{
	\io u_t^2(\cdot,t)
	+ \io u_x^2(\cdot,t)
	+ \io \Theta^2(\cdot,t) \bigg\}
	+ \int_0^T \io \Theta_x^2 
	\le \Gamma_3(c_1,T) \cdot \bigg\{
	\io u_{0t}^2
	+ \io u_{0x}^2
	+ \io \Theta_0^2
	\bigg\}.
  \eas
  Again explicitly relying on (\ref{i1}), from this we obtain (\ref{24.1}) upon an obvious selection of $\Gamma_4(K,T)$.
\qed
We are thereby placed in a position which enables us to turn Lemma \ref{lem21} into the following statement 
which actually is stronger than the part of Proposition \ref{prop98} focusing on continuous dependence.
\begin{lem}\label{lem97}
  Assume (\ref{f}), and suppose that for some $T>0$, we are given 
  sequences $(u^{(k)})_{k\in\N} \subset C^0([0,T];W_0^{1,2}(\Om))$ and
  $(\Theta^{(k)})_{k\in\N} \subset C^0([0,T];L^2(\Om)) \cap L^2((0,T);W^{1,2}(\Om))$
  with the property that $(u^{(k)}_t)_{k\in\N} \subset C^0([0,T];L^2(\Om))$,
  and that for each $k\in\N$, $(u^{(k)},\Theta^{(k)})$ is a weak solution of
  (\ref{0}) in the sense of Definition \ref{dw}, corresponding to initial data $(u_0^{(k)},u_{0t}^{(k)},\Theta_0^{(k)})$, $k\in\N$,
  which are such that $(u_0^{(k)})_{k\in\N}$, $(u_{0t}^{(k)})_{k\in\N}$ and $(\Theta_0^{(k)})_{k\in\N}$ form Cauchy sequences
  in $W_0^{1,2}(\Om), L^2(\Om)$ and $L^2(\Om)$, respectively.
  Then 
  \bea{97.1}
	& & \hs{-20mm}
	\sup_{k'>k} \sup_{t\in (0,T)} \bigg\{ 
	\io |u_t^{(k)}(\cdot,t)-u_t^{(k')}(\cdot,t)|^2
	+ \io |u_x^{(k)}(\cdot,t)-u_x^{(k')}(\cdot,t)|^2
	+ \io |\Theta^{(k)}(\cdot,t)-\Theta^{(k')}(\cdot,t)|^2
	\bigg\} \nn\\
	& & + \int_0^T \io |\Theta_x^{(k)} - \Theta_x^{(k')} |^2	
	\to 0
	\qquad \mbox{as } k\to\infty.
  \eea
\end{lem}
\proof
  Using that our assumptions particularly ensure boundedness of $((u_0^{(k)},u_{0t}^{(k)},\Theta_0^{(k)}))_{k\in\N}$ in
  $W^{1,2}(\Om)\times L^2(\Om)\times L^2(\Om)$,
  we can find $K>0$ such that (\ref{i1}) holds.
  We may thus invoke Lemma \ref{lem24} to see that if we let $\Gamma_4(\cdot,\cdot)$ be as provided there, then
  \bas
	\sup_{t\in (0,T)} \big\{ \|u_t^{(k)}(\cdot,t)\|_{L^2(\Om)} + \|\Theta^{(k)}(\cdot,t)\|_{L^1(\Om)} \big\}
	+ \int_0^T \io |\Theta^{(k)}_x|^2
	\le \Gamma_4(K,T)
	\qquad \mbox{for all } k\in\N.
  \eas
  This, in turn, suitably asserts the hypotheses of Lemma \ref{lem21}, whence taking $\Gamma_3(\cdot,\cdot)$ from that reference
  we infer that
  \bas
	& & \hs{-20mm}
	\sup_{t\in (0,T)} \bigg\{
	\io |u_t^{(k)}(\cdot,t)-u_t^{(k')}(\cdot,t)|^2
	+ \io |u_x^{(k)}(\cdot,t)-u_x^{(k')}(\cdot,t)|^2 
	+ \io |\Theta^{(k)}(\cdot,t)-\Theta^{(k')}(\cdot,t)|^2 \bigg\} \\
	& & \hs{-10mm}
	+ \int_0^T \io |\Theta^{(k)}_x-\Theta^{(k')}_x|^2  \nn\\
	&\le& \Gamma_3\big(\Gamma_4(K,T),T\big) \cdot \bigg\{
	\io |u_{0t}^{(k)}-u_{0t}^{(k')}|^2
	+ \io |u_{0x}^{(k)}-u_{0x}^{(k')}|^2
	+ \io |\Theta_0^{(k)}-\Theta_0^{(k')}|^2
	\bigg\}
  \eas
  for all $k\in\N$ and $k'\in\N$. Making full use of our hypothesis on the initial data now, we thereby obtain (\ref{97.1}).
\qed
\mysection{Existence}\label{sect4}
\subsection{Regularized problems and a general approximation property thereof}
Next concerned with the claim on existence of a solution satisfying (\ref{99.1}), in a first step toward this we 
address this issue for a class of initial data narrower than that in (\ref{init}). 
Specifically, we shall first assume that
\be{i0}
	v_0 \in C_0^\infty(\Om),
	\quad 
	u_0 \in C_0^\infty(\Om)
	\quad \mbox{and} \quad
	\Theta_0 \in C^\infty(\bom)
	\mbox{ is positive in $\bom$ and such that } \Theta_{0x} \in C_0^\infty(\Om),
\ee
and for $\eps\in (0,1)$ we shall consider the approximate variant of (\ref{0}) given by
\be{0eps}
	\lball
	\vepst = - \eps \vepsxxxx + \uepsxx - \big( f(\Teps)\big)_x,
	\qquad & x\in\Om, \ t>0, \\[1mm]
	\uepst = \eps \uepsxx + \veps,
	\qquad & x\in\Om, \ t>0, \\[1mm]
	\Tepst = \Tepsxx - f(\Teps) \vepsx,
	\qquad & x\in\Om, \ t>0, \\[1mm]
	\veps=\vepsxx=0, \quad \ueps=0, \quad \Tepsx=0,
	\qquad & x\in\pO, \ t>0,
	\ear
\ee
along with the initial conditions
\be{0epsii}
	\veps(x,0)=v_0(x), \quad \ueps(x,0)=u_0(x), \quad \Teps(x,0)=\Theta_0(x),
	\qquad  x\in\Om.
\ee
Then standard parabolic theory (\cite{amann}) can be used to see that each of these problems admits a global classical solution
with regularity properties convenient for our later purposes.
\begin{lem}\label{lem_loc}
  Assume (\ref{f}), and suppose that (\ref{i0}) holds.
  Then for each $\eps\in (0,1)$, there exist 
  \be{l1}
	\lbal
	\veps\in C^{4,1}(\bom\times [0,\infty)), \\
	\ueps\in C^{2,1}(\bom\times [0,\infty))
	\qquad \mbox{and} \\
	\Teps \in C^{2,1}(\bom\times [0,\infty))
	\ear
  \ee
  such that $\Teps>0$ in $\bom\times [0,\infty)$, and that (\ref{0eps})-(\ref{0epsii}) is satisfied in the classical pointwise sense.
  Moreover, $u_{\eps xxt}, u_{\eps xxx}, u_{\eps xxxx}$ and $\Theta_{\eps xt}$ are continuous in $\bom\times [0,\infty)$,
  and $v_{\eps xt} \in L^p_{loc}(\bom\times [0,\infty))$ for each $p\in (1,\infty)$.
\end{lem}
\proof
  For a proof of global existence of a solution fulfilling (\ref{l1}), we may refer to the reasonings 
  in \cite[Lemma 2.3]{claes_lankeit_win} and \cite[Lemma 2.4]{win_existence},
  noting that nonnegativity of $\Teps$ can be obtained by straightforward application of a comparison principle
  relying on our overall assumption that $f(0)=0$  (cf.~e.g., \cite[Lemma 2.2]{claes_win_NONRWA}), 
  and that strict positivity of $\Teps$ thereupon is a direct
  consequence of a strong maximum principle (see \cite[Theorem 10.13]{hulshof_ellpar}, for instance).
  Based on the local H\"older continuity of $f$ and $f'$ in $[0,\infty)$, as contained in the requirements in (\ref{f}),
  successive applications of parabolic Schauder theory (\cite{LSU}) 
  to the second, the first and then the third equation in (\ref{0eps})
  thereafter readily reveal the claimed continuity properties of
  $u_{\eps xxt}, u_{\eps xxx}, u_{\eps xxxx}$ and $\Theta_{\eps xt}$.
  For $p\in (1,\infty)$ and $T>0$, finally, the inclusion $v_{\eps xt} \in L^p(\Om\times (0,T))$ is a consequence of
  maximal Sobolev regularity features of linear parabolic flows, because the continuity of $f''$ and $f'$ as well as 
  of $u_{\eps xxx}, \Teps, \Tepsx$ and $\Tepsxx$ ensure
  that all the functions $u_{\eps xxx}, f''(\Teps) \Tepsx^2$ and $f'(\Teps) \Tepsxx$ are bounded in $\Om\times (0,T)$.
\qed
For fixed initial data satisfying the strong assumptions in (\ref{i0}), we may draw on an {\em a priori} estimation
procedure and an appropriate subsequence extraction developed in \cite{win_existence} to infer that these
solutions approach a global weak solution to (\ref{0}) in the following sense.
\begin{lem}\label{lem51}
  Suppose that (\ref{f}) and (\ref{i0}) hold.
  Then there exist $(\eps_j)_{j\in\N} \subset (0,1)$, as well as a global weak solution $(u,\Theta)$ of (\ref{0})
  in the sense of Definition \ref{dw}, such that $\eps_j\searrow 0$ as $j\to\infty$, and that
  \be{conv}
	\ueps \wto u,
	\quad
	\uepsx \wto u_x,
	\quad 	
	\veps \wto u_t,
	\quad
	\Teps \wto \Theta
	\quad \mbox{and} \quad
	\Tepsx \wto \Theta_x
	\qquad \mbox{in } L^2_{loc}(\bom\times [0,\infty)).
  \ee
\end{lem}
\subsection{A short-term estimate in $W^{2,2}(\Om) \times W^{1,2}(\Om) \times W^{1,2}(\Om)$}
Yet for initial data satisfying (\ref{i0}), we now intend to improve our information on regularity of the solution
gained in Lemma \ref{lem51}, where we will first concentrate on suitably small neighborhoods of the temporal origin.\abs
Our considerations in this direction will be launched by the following basic property which, parallel to a similar
observation made in \cite{cieslak_SIMA} in the particular case when $f\equiv id$, relies on a favorable 
cancellation of contributions due to the nonlinearities in (\ref{0eps}).
\begin{lem}\label{lem4}
  Assume (\ref{f}), and let
  \be{rho}
	\rho(\xi):=\frac{f'(\xi)}{f(\xi)},	
	\qquad \xi>0.
  \ee
  Then whenever (\ref{i0}) holds and $\eps\in (0,1)$, the solution of (\ref{0eps})-(\ref{0epsii}) from Lemma \ref{lem_loc} satisfies
  \bea{4.1}
	& & \hs{-20mm}
	\frac{d}{dt} \bigg\{ \frac{1}{2} \io \vepsx^2 + \frac{1}{2} \io \uepsxx^2 + \frac{1}{2} \io \rho(\Teps) \Tepsx^2 \bigg\}
	+ \io \rho(\Teps) \Tepsxx^2 
	+ \eps \io \vepsxxx^2
	+ \eps \io \uepsxxx^2 \nn\\
	&=& - \frac{1}{2} \io \rho'(\Teps) \Tepsx^2 \Tepsxx
	- \io \rho(\Teps) f'(\Teps) \Tepsx^2 \vepsx
	- \frac{1}{2} \io \rho'(\Teps) f(\Teps) \Tepsx^2 \vepsx
  \eea
  for all $t>0$.
\end{lem}
\proof
  From the first two equations in (\ref{0eps}) we obtain that thanks to the regularity features asserted by Lemma \ref{lem_loc},
  \bas
	\frac{1}{2} \frac{d}{dt} \io \vepsx^2
	&=& - \io \vepsxx \cdot \big\{ -\eps \vepsxxxx + \uepsxx - f'(\Teps) \Tepsx \big\} \nn\\
	&=& - \eps \io \vepsxxx^2 - \io \uepsxx \cdot \big\{ u_{\eps xx t} - \eps u_{\eps xxxx} \big\}
	+ \io f'(\Teps) \Tepsx \vepsxx 
  \eas
  and hence
  \be{4.2}
	\frac{1}{2} \frac{d}{dt} \io \vepsx^2 + \frac{1}{2} \frac{d}{dt} \io \uepsxx^2
	+ \eps \io \vepsxxx^2 + \eps \io \uepsxxx^2
	= \io f'(\Teps) \Tepsx \vepsxx
  \ee
  for all $t>0$.
  Likewise, in view of the definition of $\rho$ we infer from the third equation in (\ref{0eps}) that
  \bas
	\frac{1}{2} \frac{d}{dt} \io \rho(\Teps)\Tepsx^2
	&=& \io \rho(\Teps) \Tepsx \cdot \big\{ \Tepsxx - f(\Teps) \vepsx\big\}_x
	+ \frac{1}{2} \io \rho'(\Teps) \Tepsx^2 \cdot \big\{ \Tepsxx - f(\Teps) \vepsx\big\} \nn\\
	&=& - \io \rho(\Teps)\Tepsxx^2
	- \io \rho'(\Teps) \Tepsx^2 \Tepsxx \nn\\
	& & - \io \rho(\Teps)f'(\Teps) \Tepsx^2 \vepsx
	- \io \rho(\Teps) f(\Teps) \Tepsx\vepsxx \nn\\
	& & + \frac{1}{2} \io \rho'(\Teps) \Tepsx^2 \Tepsxx
	- \frac{1}{2} \io \rho'(\Teps) f(\Teps) \Tepsx^2 \vepsx
  \eas
  for all $t>0$.
  On adding this to (\ref{4.2}) we see that the crucial rightmost summand in (\ref{4.2}) can be cancelled, and that suitably
  simplifying thus leads to (\ref{4.1}).
\qed
In order to suitably control the expressions in (\ref{4.1}) containing $\rho(\Teps), \rho'(\Teps), f(\Teps)$ and $f'(\Teps)$,
we recall from \cite[Remark 6.8]{win_existence} that the two-sided estimates on $\Theta_0$
in (\ref{i0}) are inherited by solutions of (\ref{0eps})-(\ref{0epsii}) in the following sense.
\begin{lem}\label{lem5}
  Assume (\ref{f}) and (\ref{i0}).
  Then there exists $C>0$ such that for the solutions of (\ref{0eps})-(\ref{0epsii}) we have
  \be{5.1}
	\frac{1}{C} \le \Teps(x,t) \le C
	\qquad \mbox{for all $x\in\Om$, $t>0$ and } \eps\in (0,1).
  \ee
\end{lem}
Throughout adequately small time intervals, using this turns (\ref{4.1}) into $\eps$-independent estimates
in conveniently strong topological frameworks:
\begin{lem}\label{lem6}
  Assume (\ref{f}) and (\ref{i0}).
  Then there exist $\tau>0$ and $C>0$ such that for 
  any choice of $\eps\in (0,1)$,
  the solution of (\ref{0eps})-(\ref{0epsii}) has the property that
  \be{6.1}
	\io \vepsx^2(\cdot,t) \le C
	\qquad \mbox{for all } t\in (0,\tau)
  \ee
  and
  \be{6.2}
	\io \uepsxx^2(\cdot,t) \le C
	\qquad \mbox{for all } t\in (0,\tau)
  \ee
  and
  \be{6.3}
	\io \Tepsx^2(\cdot,t) \le C
	\qquad \mbox{for all } t\in (0,\tau)
  \ee
  as well as
  \be{6.4}
	\int_0^\tau \io \Tepsxx^2 \le C.
  \ee
\end{lem}
\proof
  According to Lemma \ref{lem5} and the strict positivity and boundedness of $f'$, and hence also of $f$, on $(0,\infty)$,
  we can find positive constants $c_i, i\in\{1,2,3,4,5\}$, such that whenever $\eps\in (0,1)$,
  \be{6.6}
	c_1 \le f(\Teps) \le c_2,
	\quad c_3 \le f'(\Teps) \le c_4
	\quad \mbox{and} \quad
	|f''(\Teps)| \le c_5
	\qquad \mbox{in } \Om\times (0,\infty),
  \ee
  which by definition of $\rho$ implies that if we let $c_6:=\frac{c_3}{c_2}$, 
  $c_7:=\frac{c_4}{c_1}$ and $c_8:=\frac{c_5}{c_1}+\frac{c_4^2}{c_1^2}$, then
  \be{6.7}
	c_6 \le \rho(\Teps) \le c_7
	\quad \mbox{and} \quad
	|\rho'(\Teps)| \le c_8
	\qquad \mbox{in } \Om\times (0,\infty)
  \ee
  for each $\eps\in (0,1)$.
  In particular, Lemma \ref{lem4} thus says that if for $\eps\in (0,1)$ we let
  \be{6.8}
	\yeps(t):=1+\frac{1}{2} \io \vepsx^2(\cdot,t) + \frac{1}{2} \io \uepsxx^2(\cdot,t) 
	+ \frac{1}{2} \io \rho(\Teps(\cdot,t))\Tepsx^2(\cdot,t),
	\qquad t\ge 0,
  \ee
  then
  \bas
	\yeps'(t) + c_6 \io \Tepsxx^2
	= - \frac{1}{2} \io \rho'(\Teps) \Tepsx^2 \Tepsxx
	- \io \rho(\Teps) f'(\Teps) \Tepsx^2 \vepsx 
	- \frac{1}{2} \io \rho'(\Teps) f(\Teps) \Tepsx^2 \vepsx
  \eas
  for all $t>0$,
  where by Young's inequality, (\ref{6.7}) and (\ref{6.6}),
  \bas
	- \frac{1}{2} \io \rho'(\Teps) \Tepsx^2 \Tepsxx
	&\le& \frac{c_6}{2} \io \Tepsxx^2
	+ \frac{1}{8c_6} \io \rho'^2(\Teps) \Tepsx^4 \\
	&\le& \frac{c_6}{2} \io \Tepsxx^2
	+ \frac{c_8^2}{8c_6} \io \Tepsx^4
  \eas
  and
  \bas
	- \io \rho(\Teps) f'(\Teps) \Tepsx^2 \vepsx
	&\le& \frac{1}{4} \io \vepsx^2
	+ \io \rho^2(\Teps) f'^2(\Teps) \Tepsx^4 \\
	&\le& \frac{1}{4} \io \vepsx^2
	+ c_4^2 c_7^2 \io \Tepsx^4
  \eas
  as well as
  \bas
	- \frac{1}{2} \io \rho'(\Teps) f(\Teps) \Tepsx^2 \vepsx
	&\le& \frac{1}{4} \io \vepsx^2
	+ \frac{1}{4} \io \rho'^2(\Teps) f^2(\Teps) \Tepsx^4 \\
	&\le& \frac{1}{4} \io \vepsx^2
	+ \frac{c_2^2 c_8^2}{4} \io \Tepsx^4 
  \eas
  for all $t>0$.
  Therefore, 
  \be{6.9}
	\yeps'(t)
	+ \frac{c_6}{2} \io \Tepsxx^2
	\le \frac{1}{2} \io \vepsx^2
	+ c_9 \io \Tepsx^4
	\qquad \mbox{for all } t>0
  \ee
  with $c_9:=\frac{c_8^2}{8c_6} + c_4^2 c_7^2 + \frac{c_2^2 c_8^2}{4}$.
  We now employ a Gagliardo-Nirenberg inequality to fix $c_{10}>0$ in such a way that
  \bas
	\|\vp\|_{L^4(\Om)}^4 
	\le c_{10} \|\vp_x\|_{L^2(\Om)} \|\vp\|_{L^2(\Om)}^3
	\qquad \mbox{for all $\vp\in C^1(\bom)$ fulfilling $\vp=0$ on $\pO$,}
  \eas
  whence relying on the Neumann boundary condition in (\ref{0eps}) along with Young's inequality we can estimate
  \bas
	c_9 \io \Tepsx^4
	&\le& c_9 c_{10} \|\Tepsxx\|_{L^2(\Om)} \|\Tepsx\|_{L^2(\Om)}^3 \\
	&\le& \frac{c_6}{4} \|\Tepsxx\|_{L^2(\Om)}^2 
	+ \frac{c_9^2 c_{10}^2}{c_6} \|\Tepsx\|_{L^2(\Om)}^6
	\qquad \mbox{for all } t>0.
  \eas
  In line with (\ref{6.8}), from (\ref{6.9}) we thus infer that if we write $c_{11}:=1+\frac{8c_9^2 c_{10}^2}{c_6^4}$,
  then since $\yeps\ge 1$,
  \bea{6.10}
	\yeps'(t) + \frac{c_6}{4} \io \Tepsxx^2
	&\le& \frac{1}{2} \io \vepsx^2
	+ \frac{8 c_9^2 c_{10}^2}{c_6^4} \cdot \bigg\{ \frac{c_6^3}{2} \io \Tepsx^2 \bigg\}^3 \nn\\
	&\le& \yeps(t) + \frac{8c_9^2 c_{10}^2}{c_6^4} \yeps^3(t) \nn\\[2mm]
	&\le& c_{11} \yeps^3(t)
	\qquad \mbox{for all } t>0.
  \eea
  On dropping the dissipated summand on the left-hand side herein, we particularly conclude that if we let
  \bas
	y_0:=1 + \frac{1}{2} \io v_{0x}^2
	+ \frac{1}{2} \io u_{0xx}^2
	+ \frac{1}{2} \io \rho(\Theta_0) \Theta_{0x}^2
  \eas
  as well as $\tau:=\frac{1}{4c_{11} y_0^2}$, then by an ODE comparison argument,
  \bas
	\yeps^{-2}(t) \ge \yeps^{-2}(0) - 2c_{11} t
	\ge y_0^{-2} - 2c_{11} \tau
	= \frac{1}{2} y_0^{-2}
	\qquad \mbox{for all $t\in (0,\tau)$ and } \eps\in (0,1),
  \eas
  and thus
  \be{6.11}
	\yeps(t) \le \sqrt{2} y_0
	\qquad \mbox{for all $t\in (0,\tau)$ and } \eps\in (0,1).
  \ee
  Thereupon, integrating in (\ref{6.10}) shows that
  \bas
	\frac{c_6}{4} \int_0^\tau \io \Tepsxx^2
	\le \yeps(0) + c_{11} \int_0^\tau \yeps^3(t) dt
	\le y_0 + \sqrt{8} c_{11} y_0^3
	\qquad \mbox{for all } \eps\in (0,1),
  \eas
  which together with (\ref{6.11}) establishes (\ref{6.1})-(\ref{6.4}).
\qed
\subsection{Covering arbitrary time intervals by continuous dependence on initial data}
In turning Lemma \ref{lem6} into some information on regularity of the limit $(u,\Theta)$ at a temporally global level,
we will estimate differences between solutions to (\ref{0eps}) and their time-shifted relatives in the style of (\ref{05}).
In the course of our analysis in this regard, one more application of Lemma \ref{lem13} will rely on basic bounds
that follow from the following approximate counterpart of the energy identity in (\ref{energy}).
\begin{lem}\label{lem9}
  Assume (\ref{f}) and (\ref{i0}), and let $T>0$.
  Then there exists $C(T)>0$ such that the solutions of (\ref{0eps})-(\ref{0epsii}) satisfy
  \be{9.1}
	\sup_{t\in (0,T)} \big\{ \|\veps(\cdot,t)\|_{L^2(\Om)} + \|\Teps(\cdot,t)\|_{L^1(\Om)} \big\}
	+ \int_0^T \io \Tepsx^2 \le C(T)
	\qquad \mbox{for all } \eps\in (0,1).
  \ee
\end{lem}
\proof
  By combining all three sub-problems of (\ref{0eps}) in a straighforward fashion, we obtain that
  \bas
	\frac{d}{dt} \bigg\{
	\frac{1}{2} \io \veps^2
	+ \frac{1}{2} \io \uepsx^2
	+ \io \Teps \bigg\}
	= - \eps \io \vepsxx^2
	-  \eps \io \uepsxx^2 
	\qquad \mbox{for all } \eps\in (0,1),
  \eas
  and hence particularly infer that
  \be{9.2}
	\io \veps^2(\cdot,t)
	\le c_1 := \io v_0^2 
	+ \io u_{0x}^2
	+ 2 \io \Theta_0
  \ee
  To make suitable use of this, for $\eps\in (0,1)$
  we test the third equation in (\ref{0eps}) by $\Teps$ to see that thanks to Young's inequality,
  \bas
	\frac{1}{2} \frac{d}{dt} \io \Teps^2 + \io \Tepsx^2
	&=& \io f'(\Teps) \Teps \Tepsx \veps
	+ \io f(\Teps) \Tepsx \veps \\
	&\le& \frac{1}{4} \io \Tepsx^2
	+ \io f'^2(\Teps) \Teps^2 \veps^2
	+ \frac{1}{4} \io \Tepsx^2
	+ \io f^2(\Teps) \veps^2
	\qquad \mbox{for all } t>0,
  \eas  
  so that due to (\ref{9.2}),
  \bas
	\frac{d}{dt} \io \Teps^2
	+ \io \Tepsx^2
	\le c_2 \|\Teps\|_{L^\infty(\Om)}^2
	\qquad \mbox{for all } t>0,
  \eas
  where $c_2:=2c_1 \cdot\sup_{\xi>0} \big\{ f'^2(\xi) + \frac{f^2(\xi)}{\xi^2}\big\}$ is finite due to (\ref{f}).
  Now taking $c_3>0$ such that in line with a Gagliardo-Nirenberg inequality we have
  \bas
	\|\psi\|_{L^\infty(\Om)}^2 \le c_3\|\psi_x\|_{L^2(\Om)} \|\psi\|_{L^2(\Om)}
	+ c_3 \|\psi\|_{L^2(\Om)}^2
	\qquad \mbox{for all } \psi\in C^1(\bom),
  \eas
  we can here again rely on Young's inequality in estimating
  \bas
	c_2 \|\Teps\|_{L^\infty(\Om)}^2
	\le c_2 c_3 \|\Tepsx\|_{L^2(\Om)} \|\Teps\|_{L^2(\Om)}
	+ c_2 c_3 \|\Teps\|_{L^2(\Om)}^2
	\le \frac{1}{2} \io \Tepsx^2
	+ c_4 \io \Teps^2
	\qquad \mbox{for all } t>0
  \eas
  with $c_4:=\frac{c_2^2 c_3^2}{2} + c_2 c_3$.
  Therefore,
  \be{9.3}
	\frac{d}{dt} \io \Teps^2 + \frac{1}{2} \io \Tepsx^2
	\le c_4 \io \Teps^2
	\qquad \mbox{for all } t>0,
  \ee
  whence explicitly using (\ref{i1}) now we may use a Gr\"onwall lemma to see that
  \be{9.4}
	\io \Teps^2(\cdot,t) \le c_5 := \bigg\{ \io \Theta_0^2 \bigg\} \cdot e^{c_4 T}
	\qquad \mbox{for all } t\in (0,T),
  \ee
  and that thus, by direct integration in (\ref{9.3}),
  \bas
	\frac{1}{2} \int_0^T \io \Tepsx^2
	\le \io \Theta_0^2
	+ c_4 \int_0^T \io \Teps^2
	\le \io \Theta_0^2 + c_4 c_5 T,
  \eas
  which together with (\ref{9.2}) and (\ref{9.4}) yields the claim.
\qed
Based on this and again Lemma \ref{lem13},
we can now rely on the fact that the parabolic system (\ref{0eps}) is autonomous to obtain the following 
variant of Lemma \ref{lem21}.
\begin{lem}\label{lem10}
  Assume (\ref{f}) and (\ref{i0}).
  Then for all $T>0$ there exists $C(T)>0$ such that with $((\veps,\ueps,\Teps))_{\eps\in (0,1)}$ taken from Lemma \ref{lem_loc},
  \bea{10.1}
	& & \hs{-20mm}
	\io |\veps(\cdot,t+h)-\veps(\cdot,t)|^2
	+ \io |\uepsx(\cdot,t+h)-\uepsx(\cdot,t)|^2
	+ \io |\Teps(\cdot,t+h)-\Teps(\cdot,t)|^2 \nn\\
	&\le& C \cdot \bigg\{
	\io |\veps(\cdot,h)-v_0|^2
	+ \io |\uepsx(\cdot,h)-u_{0x}|^2
	+ \io |\Teps(\cdot,h)-\Theta_0|^2 \bigg\}
  \eea
  for all $t\in (0,T), \eps\in (0,1)$ and $h\in (0,1)$.
\end{lem}
\proof
  We fix $T>0$, and then infer from Lemma \ref{lem9} that there exists $c_1=c_1(T)>0$ with the property that if
  for $h\in (0,1)$ and $\eps\in (0,1)$
  we write $(\whve,\whue,\whTe)(x,t):=(\veps,\ueps,\Teps)(x,t+h)$, $(x,t)\in\bom\times [0,\infty)$, then
  \be{33.31}
	\io \Teps(\cdot,t) \le c_1
	\quad \mbox{and} \quad
	\io \whve^2(\cdot,t) \le c_1
	\qquad \mbox{for all $t\in (0,T)$ and } \eps\in (0,1)
  \ee
  as well as
  \be{33.3}
	\int_0^T \io \Tepsx^2 \le c_1
	\qquad \mbox{for all } \eps\in (0,1).
  \ee
  Now taking any $\eps\in (0,1)$, we observe that both $(\veps,\ueps,\Teps)$ and $(\whve,\whue,\whTe)$ solve 
  the autonomous boundary value problem (\ref{0eps}), whence on linearly combining we see that
  \bas
	(\veps-\whve)_t = -\eps (\veps-\whve)_{xxxx} + (\ueps-\whue)_{xx} - \big\{ f'(\Teps)\Tepsx - f'(\whTe)\whTex \big\}
  \eas
  and
  \bas
	(\ueps-\whue)_t=\eps (\ueps-\whue)_{xx} + \veps-\whve
  \eas
  as well as
  \bas
	(\Teps-\whTe)_t 
	&=& (\Teps-\whTe)_{xx} - f(\Teps) \vepsx + f(\whTe) \whvex \nn\\
	&=& (\Teps-\whTe)_{xx} + \big\{ f'(\Teps)\Tepsx \veps - f'(\whTe)\whTex \whve\big\}
	- \pa_x \big\{ f(\Teps) \veps - f(\whTe)\whve\big\}
  \eas
  in $\Om\times (0,T)$.
  Using $\veps-\whve$, $(\ueps-\whue)_{xx}$ and $\Teps-\whTe$ as test functions here, we obtain that
  \bas
	& & \hs{-30mm}
	\frac{1}{2} \frac{d}{dt} \io (\veps-\whve)^2
	+ \eps \io (\veps-\whve)_{xx}^2 \\
	&=& \io (\ueps-\whue)_{xx}(\veps-\whve)
	- \io \big\{ f'(\Teps)\Tepsx - f'(\whTe)\whTex\big\} \cdot (\veps-\whve) \\
	&=& - \eps \io (\ueps-\whue)_{xx}^2
	- \frac{1}{2} \frac{d}{dt} \io (\uepsx-\whuex)^2 \\
	& & - \io \big\{ f'(\Teps)\Tepsx - f'(\whTe)\whTex\big\} \cdot (\veps-\whve) 
  \eas
  and
  \bas
	\frac{1}{2} \frac{d}{dt} \io (\Teps-\whTe)^2
	+ \io (\Tepsx-\whTex)^2
	&=& \io \big\{ f'(\Teps)\Tepsx\veps - f'(\whTe) \whTex \whve\big\} \cdot (\Teps-\whTe) \\
	& & + \io \big\{ f(\Teps)\veps-f(\whTe)\whve\big\}\cdot (\Tepsx-\whTex)
  \eas
  for all $t\in (0,T)$. Therefore, in particular,
  \bas
	& & \hs{-20mm}
	\frac{d}{dt} \bigg\{ 
	\frac{1}{2} \io (\veps-\whve)^2
	+ \frac{1}{2} \io (\uepsx-\whuex)^2
	+ \frac{1}{2} \io (\Teps-\whTe)^2 \bigg\}
	+ \io (\Tepsx-\whTex)^2 \nn\\
	&\le& 
	- \io \big\{ f'(\Teps)\Tepsx - f'(\whTe)\whTex\big\} \cdot (\veps-\whve) \\ 
	& & + \io \big\{ f'(\Teps)\Tepsx\veps - f'(\whTe) \whTex \whve\big\} \cdot (\Teps-\whTe) \\
	& & + \io \big\{ f(\Teps)\veps-f(\whTe)\whve\big\}\cdot (\Tepsx-\whTex)
	\qquad \mbox{for all } t\in (0,T),
  \eas
  where an application of Lemma \ref{lem13} on the basis of (\ref{33.31}) shows that if we let 
  $c_2\equiv c_2(T):=\Gamma_1(\frac{1}{2},c_1)$ with $\Gamma_1(\cdot,\cdot)$ as provided there,
  \bas
	& & \hs{-16mm}
	- \io \big\{ f'(\Teps)\Tepsx - f'(\whTe)\whTex\big\} \cdot (\veps-\whve) \\
	& & \hs{-12mm}
	+ \io \big\{ f'(\Teps)\Tepsx\veps - f'(\whTe) \whTex \whve\big\} \cdot (\Teps-\whTe)  
	+ \io \big\{ f(\Teps)\veps-f(\whTe)\whve\big\}\cdot (\Tepsx-\whTex)  \\
	&\le& \Big(\frac{1}{2}+\frac{1}{2}\Big) \io (\Tepsx-\whTex)^2
	+ (c_2+c_2) \cdot \bigg\{ 1 + \io \Tepsx^2 \bigg\} \cdot \bigg\{
		\io (\veps-\whve)^2
	+ \io (\Teps-\whTe)^2 \bigg\}
  \eas
  for all $t\in (0,T)$.
  Consequently,
  \bas
	y(t):=\io (\veps(\cdot,t)-\whve(\cdot,t))^2
	+ \io (\uepsx(\cdot,t)-\whuex(\cdot,t))^2
	+ \io (\Teps(\cdot,t)-\whTe(\cdot,t))^2,
	\qquad t\in [0,T],
  \eas
  satisfies  
  \be{33.6}
	y'(t) 
	\le 4c_2 \cdot \bigg\{ 1 + \io \Tepsx^2 \bigg\} \cdot y(t)
	\qquad \mbox{for all } t\in (0,T),
  \ee
  which implies that thanks to a Gr\"onwall lemma and (\ref{33.3}),
  \bas
	y(t)
	\le y(0) \cdot \exp \Bigg\{ 4c_2 \int_0^t \bigg\{ 1 + \io \Tepsx^2(\cdot,s) \bigg\} ds \Bigg\}
	\le y(0) \cdot e^{4c_2\cdot (T+c_1)}
	\qquad \mbox{for all } t\in (0,T),
  \eas
  and thereby yields (\ref{10.1}).
\qed
In passing to the limit in (\ref{10.1}) on the basis of the weak approximation features in (\ref{conv}), 
we will utilize the simple stability property recorded in the following.
\begin{lem}\label{lem55}
  Let $N\ge 1$, $T>0$ and $(\chi_j)_{j\in\N} \subset L^2(\Om\times (0,T);\R^N)$ and $\chi\in L^2(\Om\times (0,T);\R^N)$ 
  be such that
  \be{55.1}
	\chi_j \wto \chi
	\quad \mbox{in } L^2(\Om\times (0,T);\R^N)
	\qquad \mbox{as } j\to\infty,
  \ee
  and that 
  \be{55.2}
	\io |\chi_j(\cdot,t)|^2 \le K
	\qquad \mbox{for a.e.~} t\in (0,T)
  \ee
  with some $K>0$. Then
  \be{55.3}
	\io |\chi(\cdot,t)|^2 \le K
	\qquad \mbox{for a.e.~} t\in (0,T).
  \ee
\end{lem}
\proof
  Let $S\subset (0,T)$ be a null set such that each 
  $t\in (0,T)\sm S$ is a Lebesgue point of the integrable function 
  $(0,T)\ni t \mapsto \io |\chi(\cdot,t)|^2$.
  For arbitrary $t_0\in (0,T)\sm S$ 
  and $\del\in (0,T-t_0)$, we then let $\zeta(t):=1$ for $t\in (t_0,t_0+\del)$
  and $\zeta(t):=0$ for $t\in\R\sm (t_0,t_0+\del)$, and can then use (\ref{55.1}) and (\ref{55.2})
  to see that due to lower semicontinuity
  of the norm in $L^2(\Om\times (0,T);\R^N)$ with respect to weak convergence,
  \bea{55.4}
	& & \hs{-20mm}
	\frac{1}{\del} \int_{t_0}^{t_0+\del} \io |\chi|^2
	= \frac{1}{\del} \int_0^T \io |\zeta(t)\chi(x,t)|^2 dxdt \nn\\
	&\le&\liminf_{j\to\infty} \frac{1}{\del} \int_0^T \io |\zeta(t)\chi_j(x,t)|^2 dxdt 
	= \liminf_{j\to\infty} \frac{1}{\del} \int_{t_0}^{t_0+\del} \io |\chi_j(x,t)|^2 dxdt
	\le K,
  \eea
  because clearly $\Om\times (0,T)\ni (x,t)\mapsto \zeta(t)\chi_j(x,t)$ converges to
  $\Om\times (0,T)\ni (x,t)\mapsto \zeta(t)\chi(x,t)$ weakly in $L^2(\Om\times (0,T);\R^N)$ as $j\to\infty$.
  According to the Lebesgue point property of $t_0$, we may let $\del\searrow 0$ here to infer that
  $\io |\chi(\cdot,t_0)|^2 \le K$, so that (\ref{55.3}) results from the fact that $|S|=0$.
\qed
We can thereby make sure that under the strong assumption on the initial data made in (\ref{i0}), the weak
solution obtained in Lemma \ref{lem51} actually enjoys the regularity properties in (\ref{reg}).
\begin{lem}\label{lem111}
  Assume (\ref{f}) and (\ref{i0}). Then the problem (\ref{0}) admits at least one global weak solution $(u,\Theta)$, 
  in the sense of Definition \ref{dw}, which additionally satisfies (\ref{reg}).
\end{lem}
\proof
  We let $(\veps,\ueps,\Teps)$ denote the solution of (\ref{0eps})-(\ref{0epsii}) for $\eps\in (0,1)$,
  and employ Lemma \ref{lem6} to find $\tau>0$ such that 
  \bas
	(\veps)_{\eps\in (0,1)}
	\mbox{ is bounded in } C^0([0,\tau];W^{1,2}(\Om)),
  \eas
  that
  \bas
	(\ueps)_{\eps\in (0,1)}
	\mbox{ is bounded in } C^0([0,\tau];W^{2,2}(\Om)),
  \eas
  and that
  \bas
	(\Teps)_{\eps\in (0,1)}
	\mbox{ is bounded in } C^0([0,\tau];W^{1,2}(\Om)).
  \eas
  By means of Lemma \ref{lem51}, we thus obtain a sequence $(\eps_j)_{j\in\N} \subset (0,1)$ such that $\eps_j\searrow 0$
  as $j\to\infty$, and such that with some global weak solution $(u,\Theta)$, besides (\ref{conv}) we have
  \be{111.6}
	\ueps \to u
	\qquad \mbox{in } C^0([0,\tau];C^1(\bom))
  \ee
  and
  \be{111.5}
	\veps \to u_t
	\qquad \mbox{in } C^0(\bom\times [0,\tau])
  \ee
  and
  \be{111.7}
	\Teps \to \Theta
	\qquad \mbox{in } C^0(\bom\times [0,\tau])
  \ee
  as $\eps=\eps_j \searrow 0$.
  To derive regularity properties on large 
  time scales,
  we may rely on 
  Lemma \ref{lem9} and Lemma \ref{lem10} to find
  $(c_i(T))_{T>0} \subset (0,\infty)$, $i\in\{1,2\}$, such that 
  \be{111.77}
	\int_0^T \io \Tepsx^2 \le c_1(T)
	\qquad \mbox{for all $T>0$ and } \eps\in (0,1),
  \ee
  and that whenever $T>0$, $t\in (0,T), h\in (0,1)$ and $\eps\in (0,1)$,
  \bea{111.66}
	& & \hs{-20mm}
	\io \big|\veps(\cdot,t+h)-\veps(\cdot,t)\big|^2
	+ \io \big|\uepsx(\cdot,t+h)-\uepsx(\cdot,t)\big|^2
	+ \io \big|\Teps(\cdot,t+h)-\Teps(\cdot,t)\big|^2 \nn\\
	&\le& c_2(T) \cdot \bigg\{
	\io \big| \veps(\cdot,h)-v_0 \big|^2
	+ \io \big| \uepsx(\cdot,h)-u_{0x}\big|^2
	+ \io \big| \Teps(\cdot,h)-\Theta_0\big|^2 \bigg\}.
  \eea
  Here, (\ref{111.77}) together with (\ref{conv}) warrants that
  \be{111.78}
	\Theta_x \in L^2_{loc}(\bom\times [0,\infty)),
  \ee
  and in order to see how the continuity features in (\ref{reg}) can be inferred from (\ref{111.66}),
  we note that (\ref{111.6})-(\ref{111.7}) ensure that
  \bas
	\sup_{\eps\in (0,1)} \bigg\{
	\io \big| \veps(\cdot,h)-v_0 \big|^2
	+ \io \big| \uepsx(\cdot,h)-u_{0x}\big|^2
	+ \io \big| \Teps(\cdot,h)-\Theta_0\big|^2 \bigg\}
	\to 0
	\qquad \mbox{as } h\searrow 0.
  \eas
  whence given any $T>0$, for arbitrary $\eta>0$ we can pick $h_0(\eta,T)\in (0,1)$ such that
  for all $h\in (0,h_0(\eta,T))$ and $\eps\in (0,1)$ we have
  \bas
	\sup_{\eps\in (0,1)} \bigg\{
	\io \big| \veps(\cdot,h)-v_0 \big|^2
	+ \io \big| \uepsx(\cdot,h)-u_{0x}\big|^2
	+ \io \big| \Teps(\cdot,h)-\Theta_0\big|^2 \bigg\} 
	\le \eta.
  \eas
  Since by (\ref{conv}) we have
  \bas
	(\veps,\ueps,\Teps)(\cdot,\cdot+h) - (\veps,\ueps,\Teps)
	\wto
	(u_t,u,\Theta)(\cdot,\cdot+h) - (u_t,u,\Theta)
  \eas
  in $L^2(\Om\times (0,T))$ as $\eps=\eps_j\searrow 0$, 
  from (\ref{111.66}) and Lemma \ref{lem55} we infer that
  \bas
	& & \hs{-20mm}
	{\rm{ess}} \sup_{\hs{-4mm} t\in (0,T)} \bigg\{ 
	\io \big|u_t(\cdot,t+h)-u_t(\cdot,t)\big|^2
	+ \io \big| u_x(\cdot,t+h)-u_x(\cdot,t)\big|^2 \\
	& & 
	+ \io \big| \Theta(\cdot,t+h)-\Theta(\cdot,t)\big|^2 \bigg\}
	\le \eta
	\quad \mbox{for all } h\in \big(0,h(\eta,T)\big).
  \eas
  As $\eta>0$ and $T>0$ were arbitrary, upon modifying on a null set of times if necessary and recalling (\ref{111.78})
  we thereby obtain (\ref{reg}).
\qed
\subsection{General initial data from $W_0^{1,2}(\Om) \times L^2(\Om) \times L^2(\Om)$.
Proofs of Theorem \ref{theo99} and Proposition \ref{prop98}}
It is now sufficient to return to Lemma \ref{lem97} to cover arbitrary initial data fulfilling (\ref{init})
by means of another approximation argument, relying on Lemma \ref{lem111} and, essentially, the density
of $C_0^\infty(\Om)$ in $W_0^{1,2}(\Om)$ and $L^2(\Om)$.
\begin{lem}\label{lem11}
  Assume (\ref{f}) and (\ref{init}).
  Then the problem (\ref{0}) admits at least one global weak solution satisfying (\ref{reg}).
\end{lem}
\proof
  For $k\in\N$, we let
  \be{11.01}
	v_0^{(k)} \in C_0^\infty(\Om), 
	\quad 
	u_0^{(k)} \in C_0^\infty(\Om)
	\quad \mbox{and} \quad
	\Theta_0^{(k)} \in C^\infty(\bom)
	\mbox{ be such that $\Theta_0^{(k)} >0$ in $\bom$ and } \Theta_{0x}^{(k)} \in C_0^\infty(\Om),
  \ee
  and that
  \be{11.1}
	v_0^{(k)} \to u_{0t}
	\ \mbox{in } L^2(\Om),
	\quad
	u_0^{(k)} \to u_0
	\ \mbox{in } W^{1,2}(\Om)
	\quad \mbox{and} \quad
	\Theta_0^{(k)} \to \Theta_0
	\ \mbox{in } L^2(\Om)
	\qquad \mbox{as } k\to\infty.
  \ee
  Then due to (\ref{11.01}), from Lemma \ref{lem111} we know that for each $k\in\N$ the problem (\ref{0}) possesses
  a global weak solution $(u^{(k)},\Theta^{(k)})$
  fulfilling (\ref{reg}), whereupon Lemma \ref{lem97} asserts that thanks to (\ref{11.1})
  the collection of these solutions has the Cauchy sequence properties in (\ref{97.1}) for arbitrary $T>0$.
  By completeness, we thus obtain $u\in C^0([0,\infty);W_0^{1,2}(\Om))$ and 
  $\Theta\in C^0([0,\infty);L^2(\Om)) \cap L^2_{loc}([0,\infty);W^{1,2}(\Om))$ such that
  $u_t\in C^0([0,\infty);L^2(\Om))$ and $\Theta\ge 0$ a.e.~in $\Om\times (0,\infty)$, and that as $k\to\infty$ we have
  \be{11.2}
	u^{(k)} \to u,
	\quad 
	u_x^{(k)} \to u_x,
	\quad 
	u_t^{(k)} \to u_t
	\quad \mbox{and} \quad
	\Theta^{(k)} \to \Theta
	\qquad 
	\mbox{in } C^0_{loc}([0,\infty);L^2(\Om))
  \ee
  as well as
  \be{11.3}
	\Theta_x^{(k)} \to \Theta_x
	\qquad 
	\mbox{in } L^2_{loc}(\bom\times [0,\infty)).
  \ee
  Now for fixed $T>0$ and $k\in\N$, Definition \ref{dw} says that
  for all $\vp\in C_0^\infty(\Om\times [0,T))$, 
  \bea{11.4}
	\int_0^T \io u^{(k)} \vp_{tt}
	- \io u_{0t}^{(k)} \vp(\cdot,0)
	+ \io u_0^{(k)} \vp_t(\cdot,0)
	= -  \int_0^T \io u_x^{(k)} \vp_x
	- \int_0^T \io f'(\Theta^{(k)}) \Theta_x^{(k)} \vp,
  \eea
  and that
  for each $\phi\in C_0^\infty(\bom\times [0,T))$,
  \bea{11.5}
	& & \hs{-30mm}
	- \int_0^T \io \Theta^{(k)} \phi_t - \io \Theta_0^{(k)} \phi(\cdot,0) \nn\\
	&=& - \int_0^T \io \Theta_x^{(k)} \phi_x 
	+ \int_0^T \io f'(\Theta^{(k)}) \Theta_x^{(k)} u_t^{(k)} \phi
	+ \int_0^T \io f(\Theta^{(k)}) u_t^{(k)} \phi_x,
  \eea
  and from (\ref{11.1}), (\ref{11.2}) and (\ref{11.3}) we directly obtain that
  \bas
	\int_0^T \io u^{(k)} \vp_{tt}
	- \io u_{0t}^{(k)} \vp(\cdot,0)
	+ \io u_0^{(k)} \vp_t(\cdot,0)
	\to
	\int_0^T \io u \vp_{tt}
	- \io u_{0t} \vp(\cdot,0)
	+ \io u_0 \vp_t(\cdot,0)
  \eas
  and
  \bas
	\int_0^T \io u_x^{(k)} \vp_x
	\to \int_0^T \io u_x \vp_x,
  \eas
  and that
  \bas
	- \int_0^T \io \Theta^{(k)} \phi_t - \io \Theta_0^{(k)} \phi(\cdot,0)
	\to - \int_0^T \io \Theta \phi_t - \io \Theta_0 \phi(\cdot,0)
	\quad \mbox{and} \quad
	\int_0^T \io \Theta_x^{(k)} \phi_x 
	\to \int_0^T \io \Theta_x \phi_x
  \eas
  as $k\to\infty$.
  Moreover, the dominated convergence theorem together with the continuity and boundedness of $f'$ ensures that
  \bas
	\int_0^T \io f'(\Theta^{(k)}) \Theta_x^{(k)} \vp
	\to \int_0^T \io f'(\Theta) \Theta_x \vp
	\qquad \mbox{as } k\to\infty,
  \eas
  and that since $\Theta^{(k)}_x u_t^{(k)} \to \Theta_x u_t$ in $L^1_{loc}(\bom\times [0,\infty))$ by (\ref{11.2}) and (\ref{11.3}),
  also
  \bas
	\int_0^T \io f'(\Theta^{(k)}) \Theta_x^{(k)} u_t^{(k)} \phi
	\to \int_0^T \io f'(\Theta) \Theta_x u_t \phi
	\qquad \mbox{as } k\to\infty.
  \eas
  Once more using that $\sup_{\xi>0} \frac{f(\xi)}{\xi}<\infty$ by (\ref{f}), and that thus $f(\Theta^{(k)}) \to f(\Theta)$
  in $L^2_{loc}(\bom\times [0,\infty))$ as $k\to\infty$ according to (\ref{11.3}), we furthermore find that
  \bas
	\int_0^T \io f(\Theta^{(k)}) u_t^{(k)} \phi_x
	\to \int_0^T \io f(\Theta) u_t \phi_x
	\qquad \mbox{as } k\to\infty,
  \eas
  so that from (\ref{11.4}) and (\ref{11.5}) we conclude that, in fact, $(u,\Theta)$ is a weak solution of (\ref{0}).
\qed
Establishing our main results thereby reduces to collecting the pieces already provided:\abs
\proofc of Theorem \ref{theo99}. \quad
  The part concerning existence is covered by Lemma \ref{lem11}, while the uniqueness claim has been verified already in
  Proposition \ref{prop22}.
\qed
\proofc of Proposition \ref{prop98}. \quad
  The statement is an immediate consequence of Theorem \ref{theo99} when combined with Lemma \ref{lem97}.
\qed

\bigskip

{\bf Acknowlegements.} \quad
The author
acknowledges support of the Deutsche Forschungsgemeinschaft (Project No. 444955436).\abs
{\bf Conflict of interest statement.} \quad
The author declares that he has no conflict of interest, 
and that he has no relevant financial or non-financial interests to disclose.\abs
{\bf Data availability statement.} \quad
Data sharing is not applicable to this article as no datasets were
generated or analyzed during the current study.\abs

\end{document}